\def\bu{\bullet}
\def\marker{\>\hbox{${\vcenter{\vbox{
    \hrule height 0.4pt\hbox{\vrule width 0.4pt height 6pt
    \kern6pt\vrule width 0.4pt}\hrule height 0.4pt}}}$}\>}
\def\gpic#1{#1
     \smallskip\par\noindent{\centerline{\box\graph}} \medskip}
\documentclass[12pt]{article}
\usepackage{array,amsmath,amssymb,amsthm}  				    
\usepackage{fullpage,lineno}

\usepackage{color}

\begin{document}

\newtheorem{theorem}{Theorem}[section]
\newtheorem{lemma}[theorem]{Lemma}
\newtheorem{observation}[theorem]{Observation}
\newtheorem{corollary}[theorem]{Corollary}
\newtheorem{prop}[theorem]{Proposition}
\newtheorem{conj}[theorem]{Conjecture}
\newtheorem{claim}[theorem]{Claim}
\theoremstyle{definition}
\newtheorem{defn}[theorem]{Definition}
\newtheorem{remark}[theorem]{Remark}
\newtheorem{alg}[theorem]{Algorithm}
\def\qed{\ifhmode\unskip\nobreak\hfill$\Box$\bigskip\fi \ifmmode\eqno{Box}\fi}

\def\nul{\varnothing} 
\def\st{\colon\,}   
\def\VEC#1#2#3{#1_{#2},\ldots,#1_{#3}}
\def\VECOP#1#2#3#4{#1_{#2}#4\cdots #4 #1_{#3}}
\def\SE#1#2#3{\sum_{#1=#2}^{#3}} 
\def\PE#1#2#3{\prod_{#1=#2}^{#3}}
\def\UE#1#2#3{\bigcup_{#1=#2}^{#3}}
\def\CH#1#2{\binom{#1}{#2}} 
\def\FR#1#2{\frac{#1}{#2}}
\def\FL#1{\left\lfloor{#1}\right\rfloor} \def\FFR#1#2{\FL{\frac{#1}{#2}}}
\def\CL#1{\left\lceil{#1}\right\rceil}   \def\CFR#1#2{\CL{\frac{#1}{#2}}}
\def\Gb{\overline{G}}
\def\NN{{\mathbb N}} \def\ZZ{{\mathbb Z}} \def\QQ{{\mathbb Q}}
\def\RR{{\mathbb R}} \def\GG{{\mathbb G}} \def\FF{{\mathbb F}}

\def\B#1{{\bf #1}}      \def\R#1{{\rm #1}}
\def\I#1{{\it #1}}      \def\c#1{{\cal #1}}
\def\C#1{\left | #1 \right |}    
\def\P#1{\left ( #1 \right )}    
\def\ov#1{\overline{#1}}        \def\un#1{\underline{#1}}

\def\cD{{\mathcal D}}
\def\e{{\rm e}}
\def\la{\langle}
\def\ra{\rangle}
\def\symd{\kern-.1ex{\triangle}\kern.3ex}
\long\def\skipit#1{}


\title{$3$-Regular Graphs Are $2$-Reconstructible}

\author{
Alexandr V. Kostochka\thanks{University of Illinois at Urbana--Champaign,
Urbana IL 61801, and Sobolev Institute of Mathematics, Novosibirsk 630090,
Russia: \texttt{kostochk@math.uiuc.edu}.  Research supported in part by NSF
grants DMS-1600592 and grants 18-01-00353A and 19-01-00682  of the Russian
Foundation for Basic Research.}\,,
Mina Nahvi\thanks{University of Illinois at Urbana--Champaign,
Urbana IL 61801: \texttt{mnahvi2@illinois.edu}.}\,,
Douglas B. West\thanks{Zhejiang Normal University, Jinhua, China 321004
and University of Illinois at Urbana--Champaign, Urbana IL 61801:
\texttt{dwest@math.uiuc.edu}.
Research supported by National Natural Science Foundation of China grant NNSFC
11871439.}\,,
Dara Zirlin\thanks{University of Illinois at Urbana--Champaign, Urbana IL 61801:
\texttt{zirlin2@illinois.edu}.}
}

\date{\it Dedicated to Prof.\ Xuding Zhu on his 60th Birthday}
\maketitle

\baselineskip 16pt
\vspace{-2pc}

\begin{abstract}
A graph is $\ell$-reconstructible if it is determined by its multiset
of induced subgraphs obtained by deleting $\ell$ vertices.  We prove that
$3$-regular graphs are $2$-reconstructible.
\end{abstract}

\section{Introduction}
The {\it $k$-deck} of an $n$-vertex graph is the multiset of its $\CH nk$
induced subgraphs with $k$ vertices.  The famous Reconstruction Conjecture of
Ulam~\cite{Kel1,Ulam} asserts that when $n\ge3$, every $n$-vertex graph is
determined by its $(n-1)$-deck.  In 1957, Kelly~\cite{Kel2} extended the
conjecture, considering deletion of more than one vertex.  A graph or graph
property is {\it $\ell$-reconstructible} if it is determined by the deck
obtained by deleting $\ell$ vertices.  Kelly conjectured that for each $\ell$
there is a threshold $M_\ell$ such that every graph with at least $M_\ell$
vertices is $\ell$-reconstructible.

It is thought that perhaps $M_2=6$ (McMullen and Radziszowski~\cite{MR}
conjectured $M_\ell\le3\ell$).  Since the graph $C_4+K_1$ and the tree
$K_{1,3}'$ obtained by subdividing one edge of $K_{1,3}$ have the same
$3$-deck, $M_2\ge6$.  Spinoza and West~\cite{SW} showed that $P_{2\ell}$ and
$C_{\ell+1}+P_{\ell-1}$ have the same $\ell$-deck~\cite{SW}, and hence
$M_\ell\ge2\ell+1$.

Let $\cD_k(G)$ denote the $k$-deck of a graph $G$.  The elements of 
$\cD_k(G)$ are called {\it $k$-cards} or just {\it cards}.
Fix $n$ to be the number of vertices of the graph $G$ whose $k$-deck we
are given, so that $\C{\cD_k(G)}=\CH nk$.  Since
every $(k-1)$-card appears in exactly $n-k+1$ of the $k$-cards, always
$\cD_k(G)$ determines $\cD_{k-1}(G)$.
It is thus sensible to define the {\it reconstructibility} of a graph $G$
to be the maximum $\ell$ such that $G$ is $\ell$-reconstructible.
Spinoza and West~\cite{SW} determined the reconstructibility of all graphs
with maximum degree at most $2$.  They also showed that almost all graphs are
${(1-o(1))n/2}$-reconstructible, extending the observations
in~\cite{Bol,Chi,Mul} that almost all graphs are $1$-reconstructible.

Much research in graph reconstruction has focused on finding classes or
properties of graphs that are $1$-reconstructible.  In the spirit of 
Kelly's Conjecture, we ask what can be shown to be $\ell$-reconstructible
for larger $\ell$.  Initial attention has considered the degree list, which
trivially is $1$-reconstructible because the $2$-deck already determines the
number of edges.  Chernyak~\cite{Che} showed that the degree list is
$2$-reconstructible when $n\ge6$ (sharp by $\{C_4+K_1,K'_{1,3}\}$).  The
present authors~\cite{KNWZ} showed that the degree list is $3$-reconstructible
when $n\ge7$ (sharp by $\{C_5+K_1,K''_{1,3}\}$, where $K''_{1,3}$ is the tree
obtained from $K_{1,3}$ by subdividing two edges).  For $\ell$ in general,
Taylor~\cite{Tay} showed that the degree list is $\ell$-reconstructible when
$n\ge \e\ell+O(\log\ell)$, where $\e$ is the base of the natural logarithm.

In one of the first results on reconstruction, Kelly~\cite{Kel1} proved that
disconnected graphs are $1$-reconstructible.  Manvel~\cite{Man} proved that
disconnected graphs having no component with $n-1$ vertices are
$2$-reconstructible.  He also observed that $2$-reconstructibility of 
disconnected graphs consisting of one isolated vertex and a connected graph
is equivalent to the original Reconstruction Conjecture that all graphs are
$1$-reconstructible when $n\ge3$.  

In addition, Manvel~\cite{Man} showed that whether an $n$-vertex graph is
connected can be determined from its $(n-2)$-deck when $n\ge6$.  That is,
connectedness is $2$-reconstructible when $n\ge6$ (sharp by
$\{C_4+K_1,K_{1,3}'\}$).  The present authors~\cite{KNWZ} showed that
connectedness is $3$-reconstructible when $n\ge7$ (sharp by
$\{C_5+K_1,K_{1,3}''\}$).  Spinoza and West~\cite{SW} showed that connectedness
is $\ell$-reconstructible when $n>2\ell^{(\ell+1)^2}$ (this is not sharp).  

Since the degree list is $1$-reconstructible, regular graphs are
$1$-reconstructible (after determining that the deck arises only from
$r$-regular graphs, make the missing vertex adjacent to the $r$ vertices of
degree $r-1$ in any card).  At a meeting in Sanya in 2019, Bojan Mohar asked
whether regular graphs are $2$-reconstructible.  This is not immediate, even
though the degree list is $2$-reconstructible, because we must determine which
of the deficient vertices is adjacent to which of the two missing vertices.
In this paper, we prove the following result.

\begin{theorem}\label{main}
Every $3$-regular graph is $2$-reconstructible.
\end{theorem}

A useful property of $3$-regular graphs not shared by regular graphs of higher
degree is that any two cycles through a vertex have a common edge.  Lacking
this property, it seems difficult to extend our approach to regular graphs of
higher degree.

\section{Preliminaries}

Let $\cD$ be the $(n-2)$-deck of a $3$-regular graph with $n$ vertices
(henceforth we simply say {\it deck} for the $(n-2)$-deck).
A {\it reconstruction} (from $\cD$) is an $n$-vertex graph whose deck is $\cD$.
Since $K_4$ is determined by its $2$-deck and $n$ must be even, we may assume
$n\ge6$.  Now the $2$-reconstructibility of the degree list implies that every
reconstruction is $3$-regular.

We aim to prove that there is only one possible reconstruction, or equivalently
that all reconstructions are isomorphic.  To do this, we will restrict the
properties of an arbitrary reconstruction from a deck that has nonisomorphic
reconstructions.  We will repeatedly (often implicitly) use the following
trivial observation.

\begin{observation}\label{altsame}
If $H$ is an alternative reconstruction from a card in the deck of a 
$3$-regular graph $G$, then $H$ satisfies all properties that have been
shown to hold for every reconstruction from a deck that has more than one
$3$-regular reconstruction.
\end{observation}

Here $H$ witnesses that $G$ is a counterexample to Theorem~\ref{main}.  
Our first restriction on the properties of such a graph $G$ arose in discussion
with Martin Merker, Bojan Mohar, and Hehui Wu.
Let a {\it $j$-vertex} be a vertex of degree $j$.

\begin{lemma}\label{girth5}
Given a card obtained by deleting adjacent vertices of $G$, in every
reconstruction the missing vertices are adjacent.  Given a card obtained by
deleting vertices with a common neighbor, in every reconstruction the missing
vertices are both adjacent to that common neighbor.  Finally, every
reconstruction has girth at least $5$.
\end{lemma}
\begin{proof}
The first two remarks hold because every reconstruction from the deck is
$3$-regular.

If $G$ has a triangle $T$, then a card obtained by deleting two vertices of $T$
has two $1$-vertices or has one $1$-vertex and two $2$-vertices.  In a
$3$-regular reconstruction, the two missing vertices must be adjacent and must
both be adjacent to any $1$-vertex.  If the card has two $2$-vertices, then the
$2$-vertices must each be adjacent to one missing vertex.  The two
reconstructions are isomorphic, preventing $H\not\cong G$.

If $G$ has a $4$-cycle (and no triangle), then a card obtained by deleting two
nonadjacent vertices on a $4$-cycle has two $1$-vertices and two $2$-vertices.
In a $3$-regular reconstruction, the two missing vertices must both be adjacent
to both $1$-vertices, and each must be adjacent to one of the $2$-vertices.
The two reconstructions are isomorphic, preventing $H\not\cong G$.
\end{proof}

With girth at least $5$, we have $n\ge10$.
Next we note the analogue of Kelly's Lemma~\cite{Kel2}.

\begin{lemma}\label{subcount}
For each graph $F$ with at most $n-\ell$ vertices, the number of subgraphs of 
$G$ isomorphic to $F$ is $\ell$-reconstructible.  In particular, the number
of cycles of any length at most $n-2$ is $2$-reconstructible.
\end{lemma}
\begin{proof}
Each copy of $F$ appears in exactly $\CH{n-\C{V(F)}}\ell$ cards.
\end{proof}

A elementary exercise states that every $n$-vertex graph with at
least $n+1$ edges has girth at most $\CL{2n/3}$.  With $\CL{2n/3}\le n-3$
when $n\ge9$, Lemma~\ref{subcount} yields the following.

\begin{corollary}\label{girth}
Every reconstruction has the same girth $g$, the same number of $g$-cycles,
and the same number of $(g+1)$-cycles.
\hfill$\square$
\end{corollary}

Let $d_G(x,y)$ denote the distance between $x$ and $y$ in $G$.
\begin{lemma}\label{twopair}
Let $\cD$ be the deck of a $3$-regular graph $G$ with another reconstruction.
Fix $F=G-\{x,y\}\in\cD$.  If $d_G(x,y)=1$, then $F$ has four $2$-vertices.  If
$d_G(x,y)=2$, then $F$ has one $1$-vertex and four $2$-vertices.  Also, we can
recognize when $d_G(x,y)$ is $1$ or $2$ or larger.
\end{lemma}
\begin{proof}
The degree claims follow from $G$ being $3$-regular with girth at least $5$.
Since $F$ has six $2$-vertices when $d_G(x,y)>2$, we can recognize $d_G(x,y)$
being $1$ or $2$ or greater than $2$.
\end{proof}

We will usually consider cards in which the deleted vertices are at distance
at most $2$ and lie together on a shortest cycle.  We say that two $2$-vertices
in a card $F$ are {\it paired} in a reconstruction from $F$ when they have one
of the missing vertices as a common neighbor.

\begin{lemma}\label{shortest}
If $d_G(x,y)\le2$ and $x$ and $y$ both lie on a shortest cycle $C$ in $G$ (with
$g=\C{V(C)}$), then $G-\{x,y\}$ has only one alternative reconstruction, $H$.
In $H$, the missing vertices $x'$ and $y'$ complete a copy $C'$ of $C$ obtained
by substituting $x'$ for $x$ and $y'$ for $y$.  If also $xy\in E(C)$, then the
number of $g$-cycles using one or both of $\{x,',y'\}$ in $H$ is the same as
the number of $g$-cycles using one or both of $\{x,y\}$ in $G$, respectively.
\end{lemma}
\begin{proof}
When $d_G(x,y)\le2$, the four $2$-vertices in $G-\{x,y\}$ must form two pairs
in any reconstruction: two neighbors of $x'$ and two neighbors of $y'$.  There
are three ways to pair four vertices.  However, two of those $2$-vertices lie
on $C$, with the path joining them through $x$ and $y$ having length $3$ or
$4$.  Pairing them as neighbors of one of $\{x',y'\}$ creates a shorter cycle.
Since Lemma~\ref{subcount} provides the girth of $G$, this alternative pairing
is forbidden, leaving only $G$ and one alternative.

If $xy\in E(G)$, then in any reconstruction the two vertices that were adjacent
to $x$ and $y$ on any shortest $C$ must each be adjacent to one of $\{x',y'\}$
(and not the same one); otherwise a shorter cycle is formed (see
Figure~\ref{gcycle}).  Hence in every reconstruction from $G-\{x,y\}$ the
number of shortest cycles that use both missing vertices is the same.

Since we know the number of $g$-cycles in $G-\{x,y\}$, we know the number of
shortest cycles that were destroyed.  Hence we now also know the number of
$g$-cycles that use exactly one of the two missing vertices.
\end{proof}

\vspace{-1pc}

\begin{figure}[hbt]
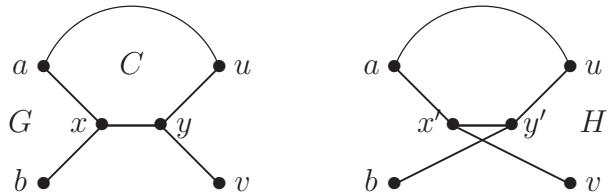

\gpic{
\expandafter\ifx\csname graph\endcsname\relax \csname newbox\endcsname\graph\fi
\expandafter\ifx\csname graphtemp\endcsname\relax \csname newdimen\endcsname\graphtemp\fi
\setbox\graph=\vtop{\vskip 0pt\hbox{%
    \graphtemp=.5ex\advance\graphtemp by 0.625in
    \rlap{\kern 0.429in\lower\graphtemp\hbox to 0pt{\hss $\bu$\hss}}%
    \graphtemp=.5ex\advance\graphtemp by 0.625in
    \rlap{\kern 0.735in\lower\graphtemp\hbox to 0pt{\hss $\bu$\hss}}%
    \graphtemp=.5ex\advance\graphtemp by 0.932in
    \rlap{\kern 0.122in\lower\graphtemp\hbox to 0pt{\hss $\bu$\hss}}%
    \graphtemp=.5ex\advance\graphtemp by 0.932in
    \rlap{\kern 1.041in\lower\graphtemp\hbox to 0pt{\hss $\bu$\hss}}%
    \graphtemp=.5ex\advance\graphtemp by 0.319in
    \rlap{\kern 0.122in\lower\graphtemp\hbox to 0pt{\hss $\bu$\hss}}%
    \graphtemp=.5ex\advance\graphtemp by 0.319in
    \rlap{\kern 1.041in\lower\graphtemp\hbox to 0pt{\hss $\bu$\hss}}%
    \special{pn 11}%
    \special{pa 429 625}%
    \special{pa 735 625}%
    \special{fp}%
    \special{pa 429 625}%
    \special{pa 122 932}%
    \special{fp}%
    \special{pa 429 625}%
    \special{pa 122 319}%
    \special{fp}%
    \special{pa 735 625}%
    \special{pa 429 625}%
    \special{fp}%
    \special{pa 735 625}%
    \special{pa 1041 932}%
    \special{fp}%
    \special{pa 735 625}%
    \special{pa 1041 319}%
    \special{fp}%
    \special{pn 8}%
    \special{ar 582 490 490 490 -2.786171 -0.355421}%
    \graphtemp=.5ex\advance\graphtemp by 0.319in
    \rlap{\kern 0.582in\lower\graphtemp\hbox to 0pt{\hss $C$\hss}}%
    \graphtemp=.5ex\advance\graphtemp by 0.625in
    \rlap{\kern 0.000in\lower\graphtemp\hbox to 0pt{\hss $G$\hss}}%
    \graphtemp=.5ex\advance\graphtemp by 0.625in
    \rlap{\kern 0.306in\lower\graphtemp\hbox to 0pt{\hss $x$\hss}}%
    \graphtemp=.5ex\advance\graphtemp by 0.932in
    \rlap{\kern 0.000in\lower\graphtemp\hbox to 0pt{\hss $b$\hss}}%
    \graphtemp=.5ex\advance\graphtemp by 0.319in
    \rlap{\kern 0.000in\lower\graphtemp\hbox to 0pt{\hss $a$\hss}}%
    \graphtemp=.5ex\advance\graphtemp by 0.625in
    \rlap{\kern 0.857in\lower\graphtemp\hbox to 0pt{\hss $y$\hss}}%
    \graphtemp=.5ex\advance\graphtemp by 0.932in
    \rlap{\kern 1.163in\lower\graphtemp\hbox to 0pt{\hss $v$\hss}}%
    \graphtemp=.5ex\advance\graphtemp by 0.319in
    \rlap{\kern 1.163in\lower\graphtemp\hbox to 0pt{\hss $u$\hss}}%
    \graphtemp=.5ex\advance\graphtemp by 0.625in
    \rlap{\kern 2.265in\lower\graphtemp\hbox to 0pt{\hss $\bu$\hss}}%
    \graphtemp=.5ex\advance\graphtemp by 0.625in
    \rlap{\kern 2.571in\lower\graphtemp\hbox to 0pt{\hss $\bu$\hss}}%
    \graphtemp=.5ex\advance\graphtemp by 0.932in
    \rlap{\kern 1.959in\lower\graphtemp\hbox to 0pt{\hss $\bu$\hss}}%
    \graphtemp=.5ex\advance\graphtemp by 0.932in
    \rlap{\kern 2.878in\lower\graphtemp\hbox to 0pt{\hss $\bu$\hss}}%
    \graphtemp=.5ex\advance\graphtemp by 0.319in
    \rlap{\kern 1.959in\lower\graphtemp\hbox to 0pt{\hss $\bu$\hss}}%
    \graphtemp=.5ex\advance\graphtemp by 0.319in
    \rlap{\kern 2.878in\lower\graphtemp\hbox to 0pt{\hss $\bu$\hss}}%
    \special{pn 11}%
    \special{pa 2265 625}%
    \special{pa 2571 625}%
    \special{fp}%
    \special{pa 2265 625}%
    \special{pa 2878 932}%
    \special{fp}%
    \special{pa 2265 625}%
    \special{pa 1959 319}%
    \special{fp}%
    \special{pa 2571 625}%
    \special{pa 2265 625}%
    \special{fp}%
    \special{pa 2571 625}%
    \special{pa 1959 932}%
    \special{fp}%
    \special{pa 2571 625}%
    \special{pa 2878 319}%
    \special{fp}%
    \special{pn 8}%
    \special{ar 2418 490 490 490 -2.786171 -0.355421}%
    \graphtemp=.5ex\advance\graphtemp by 0.319in
    \rlap{\kern 2.418in\lower\graphtemp\hbox to 0pt{\hss $~$\hss}}%
    \graphtemp=.5ex\advance\graphtemp by 0.625in
    \rlap{\kern 3.000in\lower\graphtemp\hbox to 0pt{\hss $H$\hss}}%
    \graphtemp=.5ex\advance\graphtemp by 0.625in
    \rlap{\kern 2.143in\lower\graphtemp\hbox to 0pt{\hss $x'$\hss}}%
    \graphtemp=.5ex\advance\graphtemp by 0.932in
    \rlap{\kern 1.837in\lower\graphtemp\hbox to 0pt{\hss $b$\hss}}%
    \graphtemp=.5ex\advance\graphtemp by 0.319in
    \rlap{\kern 1.837in\lower\graphtemp\hbox to 0pt{\hss $a$\hss}}%
    \graphtemp=.5ex\advance\graphtemp by 0.625in
    \rlap{\kern 2.694in\lower\graphtemp\hbox to 0pt{\hss $y'$\hss}}%
    \graphtemp=.5ex\advance\graphtemp by 0.932in
    \rlap{\kern 3.000in\lower\graphtemp\hbox to 0pt{\hss $v$\hss}}%
    \graphtemp=.5ex\advance\graphtemp by 0.319in
    \rlap{\kern 3.000in\lower\graphtemp\hbox to 0pt{\hss $u$\hss}}%
    \hbox{\vrule depth1.054in width0pt height 0pt}%
    \kern 3.000in
  }%
}%
}
\vspace{-1pc}
\caption{An edge on a shortest cycle.}\label{gcycle}
\end{figure}

\section{Configurations of Short Cycles}

Our approach to prohibiting $3$-regular graphs with alternative reconstructions
is to prohibit short cycles with common or adjacent vertices in such graphs.
With $g$ being the girth of $G$ (already $g\ge5$ by Lemma~\ref{girth5}), we
will eventually forbid having two $g$-cycles sharing an edge or connected by an
edge, and we will forbid having a $g$-cycle and a $g'$-cycle sharing an edge,
where henceforth $g'=g+1$.  These exclusions lead to a final contradiction,
because we will also show that a $g$-cycle must share an edge with some
$g'$-cycle.

Throughout this section, $G$ is a $3$-regular $n$-vertex graph whose deck
($(n-2)$-deck) $\cD$ is also the deck of some $3$-regular graph $H$ not
isomorphic to $G$.  The statements we prove restrict the structure of an
arbitrary reconstruction $G$ from $\cD$, but once proved they hold also for
an alternative reconstruction $H$ in all subsequent steps, as formalized in
Lemma~\ref{altsame}.  Hence we do not mention $G$ in the statements of the
lemmas.  Also, when some reconstruction has the assumed property, we can always
find a card as described by looking at all reconstructions from each card in
the given deck.

\begin{lemma}\label{step0}
Two $g$-cycles cannot share two consecutive edges.
A $g$-cycle and a $g'$-cycle cannot share three consecutive edges.
\end{lemma}
\begin{proof}
Let $C$ and $D$ be a $g$-cycle and a cycle of length at most $g'$ in $G$ such
that $C\cap D$ has a component $P$ with at least two edges.  Let $x$ be an
endpoint of $P$, with $xy\in E(P)$.  Let $a$ and $b$ be the other neighbors of
$x$ on $C$ and $D$, respectively.  Let $u$ and $v$ be the neighbors of $y$
other than $x$, with $yu\in E(C)\cap E(D)$.  To avoid being $G$, the
alternative reconstruction $H$ from $G-\{x,y\}$ must not pair $a$ and $b$; to
avoid having a shorter cycle, it must not pair $a$ and $u$.  Hence it pairs $a$
and $v$, and we may assume $ax',vx',y'u,y'b\in E(H)$.  Now $H$ has a $g$-cycle
$C'$ obtained from $C$ by substituting $x'$ for $x$ and $y'$ for $y$ (see
Figure~\ref{splice}).

Also $H$ has a cycle $D'$ consisting of the path $\la b,y',u\ra$ and the
$u,b$-path along $D$ that does not use $x$.  This cycle $D'$ is shorter than
$D$.  If $D$ has length $g$, this is a contradiction.  If $D$ has length $g'$
and $P$ has a third edge, then $C'$ and $D'$ are two $g$-cycles with two
consecutive common edges, which the first case already forbids for all
reconstructions.
\end{proof}

\vspace{-1pc}
\begin{figure}[hbt]
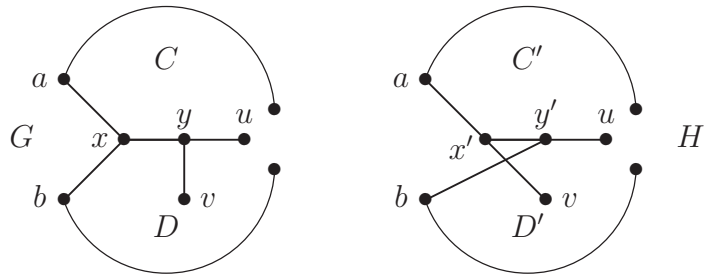

\gpic{
\expandafter\ifx\csname graph\endcsname\relax \csname newbox\endcsname\graph\fi
\expandafter\ifx\csname graphtemp\endcsname\relax \csname newdimen\endcsname\graphtemp\fi
\setbox\graph=\vtop{\vskip 0pt\hbox{%
    \graphtemp=.5ex\advance\graphtemp by 0.698in
    \rlap{\kern 0.536in\lower\graphtemp\hbox to 0pt{\hss $\bu$\hss}}%
    \graphtemp=.5ex\advance\graphtemp by 0.698in
    \rlap{\kern 0.851in\lower\graphtemp\hbox to 0pt{\hss $\bu$\hss}}%
    \graphtemp=.5ex\advance\graphtemp by 1.014in
    \rlap{\kern 0.221in\lower\graphtemp\hbox to 0pt{\hss $\bu$\hss}}%
    \graphtemp=.5ex\advance\graphtemp by 1.014in
    \rlap{\kern 0.851in\lower\graphtemp\hbox to 0pt{\hss $\bu$\hss}}%
    \graphtemp=.5ex\advance\graphtemp by 0.383in
    \rlap{\kern 0.221in\lower\graphtemp\hbox to 0pt{\hss $\bu$\hss}}%
    \graphtemp=.5ex\advance\graphtemp by 0.698in
    \rlap{\kern 1.167in\lower\graphtemp\hbox to 0pt{\hss $\bu$\hss}}%
    \graphtemp=.5ex\advance\graphtemp by 0.856in
    \rlap{\kern 1.324in\lower\graphtemp\hbox to 0pt{\hss $\bu$\hss}}%
    \graphtemp=.5ex\advance\graphtemp by 0.541in
    \rlap{\kern 1.324in\lower\graphtemp\hbox to 0pt{\hss $\bu$\hss}}%
    \special{pn 11}%
    \special{pa 536 698}%
    \special{pa 851 698}%
    \special{fp}%
    \special{pa 536 698}%
    \special{pa 221 1014}%
    \special{fp}%
    \special{pa 536 698}%
    \special{pa 221 383}%
    \special{fp}%
    \special{pa 851 698}%
    \special{pa 536 698}%
    \special{fp}%
    \special{pa 851 698}%
    \special{pa 851 1014}%
    \special{fp}%
    \special{pa 851 698}%
    \special{pa 1167 698}%
    \special{fp}%
    \special{pn 8}%
    \special{ar 757 568 568 568 -2.810165 -0.047634}%
    \special{ar 757 829 568 568 0.047634 2.810165}%
    \graphtemp=.5ex\advance\graphtemp by 0.288in
    \rlap{\kern 0.757in\lower\graphtemp\hbox to 0pt{\hss $C$\hss}}%
    \graphtemp=.5ex\advance\graphtemp by 1.171in
    \rlap{\kern 0.757in\lower\graphtemp\hbox to 0pt{\hss $D$\hss}}%
    \graphtemp=.5ex\advance\graphtemp by 0.698in
    \rlap{\kern 0.000in\lower\graphtemp\hbox to 0pt{\hss $G$\hss}}%
    \graphtemp=.5ex\advance\graphtemp by 0.698in
    \rlap{\kern 0.410in\lower\graphtemp\hbox to 0pt{\hss $x$\hss}}%
    \graphtemp=.5ex\advance\graphtemp by 1.014in
    \rlap{\kern 0.095in\lower\graphtemp\hbox to 0pt{\hss $b$\hss}}%
    \graphtemp=.5ex\advance\graphtemp by 0.383in
    \rlap{\kern 0.095in\lower\graphtemp\hbox to 0pt{\hss $a$\hss}}%
    \graphtemp=.5ex\advance\graphtemp by 0.572in
    \rlap{\kern 0.851in\lower\graphtemp\hbox to 0pt{\hss $y$\hss}}%
    \graphtemp=.5ex\advance\graphtemp by 1.014in
    \rlap{\kern 0.977in\lower\graphtemp\hbox to 0pt{\hss $v$\hss}}%
    \graphtemp=.5ex\advance\graphtemp by 0.572in
    \rlap{\kern 1.167in\lower\graphtemp\hbox to 0pt{\hss $u$\hss}}%
    \graphtemp=.5ex\advance\graphtemp by 0.698in
    \rlap{\kern 2.428in\lower\graphtemp\hbox to 0pt{\hss $\bu$\hss}}%
    \graphtemp=.5ex\advance\graphtemp by 0.698in
    \rlap{\kern 2.743in\lower\graphtemp\hbox to 0pt{\hss $\bu$\hss}}%
    \graphtemp=.5ex\advance\graphtemp by 1.014in
    \rlap{\kern 2.113in\lower\graphtemp\hbox to 0pt{\hss $\bu$\hss}}%
    \graphtemp=.5ex\advance\graphtemp by 1.014in
    \rlap{\kern 2.743in\lower\graphtemp\hbox to 0pt{\hss $\bu$\hss}}%
    \graphtemp=.5ex\advance\graphtemp by 0.383in
    \rlap{\kern 2.113in\lower\graphtemp\hbox to 0pt{\hss $\bu$\hss}}%
    \graphtemp=.5ex\advance\graphtemp by 0.698in
    \rlap{\kern 3.059in\lower\graphtemp\hbox to 0pt{\hss $\bu$\hss}}%
    \graphtemp=.5ex\advance\graphtemp by 0.856in
    \rlap{\kern 3.216in\lower\graphtemp\hbox to 0pt{\hss $\bu$\hss}}%
    \graphtemp=.5ex\advance\graphtemp by 0.541in
    \rlap{\kern 3.216in\lower\graphtemp\hbox to 0pt{\hss $\bu$\hss}}%
    \special{pn 11}%
    \special{pa 2428 698}%
    \special{pa 2743 698}%
    \special{fp}%
    \special{pa 2428 698}%
    \special{pa 2743 1014}%
    \special{fp}%
    \special{pa 2428 698}%
    \special{pa 2113 383}%
    \special{fp}%
    \special{pa 2743 698}%
    \special{pa 2428 698}%
    \special{fp}%
    \special{pa 2743 698}%
    \special{pa 2113 1014}%
    \special{fp}%
    \special{pa 2743 698}%
    \special{pa 3059 698}%
    \special{fp}%
    \special{pn 8}%
    \special{ar 2649 568 568 568 -2.810165 -0.047634}%
    \special{ar 2649 829 568 568 0.047634 2.810165}%
    \graphtemp=.5ex\advance\graphtemp by 0.288in
    \rlap{\kern 2.649in\lower\graphtemp\hbox to 0pt{\hss $C'$\hss}}%
    \graphtemp=.5ex\advance\graphtemp by 1.171in
    \rlap{\kern 2.649in\lower\graphtemp\hbox to 0pt{\hss $D'$\hss}}%
    \graphtemp=.5ex\advance\graphtemp by 0.698in
    \rlap{\kern 3.500in\lower\graphtemp\hbox to 0pt{\hss $H$\hss}}%
    \graphtemp=.5ex\advance\graphtemp by 0.761in
    \rlap{\kern 2.302in\lower\graphtemp\hbox to 0pt{\hss $x'$\hss}}%
    \graphtemp=.5ex\advance\graphtemp by 1.014in
    \rlap{\kern 1.986in\lower\graphtemp\hbox to 0pt{\hss $b$\hss}}%
    \graphtemp=.5ex\advance\graphtemp by 0.383in
    \rlap{\kern 1.986in\lower\graphtemp\hbox to 0pt{\hss $a$\hss}}%
    \graphtemp=.5ex\advance\graphtemp by 0.572in
    \rlap{\kern 2.743in\lower\graphtemp\hbox to 0pt{\hss $y'$\hss}}%
    \graphtemp=.5ex\advance\graphtemp by 1.014in
    \rlap{\kern 2.869in\lower\graphtemp\hbox to 0pt{\hss $v$\hss}}%
    \graphtemp=.5ex\advance\graphtemp by 0.572in
    \rlap{\kern 3.059in\lower\graphtemp\hbox to 0pt{\hss $u$\hss}}%
    \hbox{\vrule depth1.396in width0pt height 0pt}%
    \kern 3.500in
  }%
}%
}
\vspace{-1pc}
\caption{Consecutive edges shared by short cycles.}\label{splice}
\end{figure}

We refer to two cycles sharing two consecutive edges as {\it spliced} cycles.
We have now forbidden spliced $g$-cycles from $3$-regular graphs whose decks
have alternative reconstructions (a spliced $g$-cycle and $g'$-cycle remain
allowed, but they can't share three consecutive edges).

\begin{remark}\label{alt}
Henceforth, when $\la a,x,y,u\ra$ is a path along a $g$-cycle $C$, and
the third neighbors of $x$ and $y$ are $b$ and $v$, respectively, the
arguments we have made imply that any alternative reconstruction $H$
from $G-\{x,y\}$ is obtained by adding the vertices $x'$ and $y'$ with
$N_H(x')=\{y',a,v\}$ and $N_H(y')=\{x',u,b\}$.
\hfill$\square$
\end{remark}

When $xy$ is an edge on a $g$-cycle, Lemma~\ref{step0} implies that $x$ and $y$
each lie in at most one cycle not containing the other, since any such cycle
uses both incident edges other than $xy$.

\begin{lemma}\label{step1}
If $xy$ is an edge in two $g$-cycles, then $x$ and $y$ cannot each lie in
a $g$-cycle not containing the other.
\end{lemma}
\begin{proof}
Let $C$ and $D$ be $g$-cycles containing $xy$, with
$\la a,x,y,u\ra$ along $C$ and $\la b,x,y,v\ra$ along $D$.
By Lemma~\ref{step0}, these six vertices are distinct.  Let $H$ be the
alternative reconstruction from $G-\{x,y\}$ as in Remark~\ref{alt}.
Note that $H$ has two $g$-cycles through $xy$ (see Figure~\ref{manyg}).

Suppose that each of $x$ and $y$ lies in a $g$-cycle not containing the other.
Since $G$ has no spliced $g$-cycles, these two $g$-cycles $Q$ and $R$ pass
through $\la a',a,x,b,b'\ra$ and $\la u',u,y,v,v'\ra$, respectively, where $w'$
for $w\in\{a,b,u,v\}$ is the neighbor of $w$ not in $C\cup D$.
By Lemma~\ref{shortest}, in $H$ each of $x'$ and $y'$ lies in a $g$-cycle not
containing the other.  To avoid spliced $g$-cycles in $H$, these $g$-cycles
$Q'$ and $R'$ must pass through $\la a',a,x',v,v'\ra$ and
$\la u',u,y',b,b'\ra$, respectively.  In particular, $H$ and $G$ contain
an $a',v'$-path $P$ of length $g-4$.

Now consider $G-\{a,x\}$.  Since $ax$ lies in the $g$-cycle $C$, an alternative
reconstruction $H'$ replacing $\{a,x\}$ with $\{a'',x''\}$ can be assumed to
have the $g$-cycle $C''$ through $\la z,a'',x'',y\ra$, where $z$ is the
neighbor of $a$ on $C$ other than $x$, and the edges $a'x''$ and $a''b$ (see
Remark~\ref{alt}).  In $H'$, the path $\la a',x'',y,v,v'\ra$ combines with $P$
to form a $g$-cycle.  However, this $g$-cycle shares consecutive edges $yv$
and $vv'$ with $R$, creating spliced $g$-cycles, which is forbidden.
\end{proof}

\vspace{-1pc}
\begin{figure}[hbt]
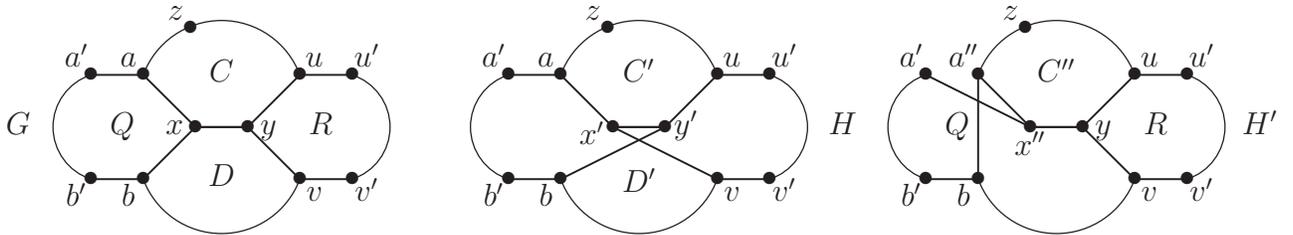

\gpic{
\expandafter\ifx\csname graph\endcsname\relax \csname newbox\endcsname\graph\fi
\expandafter\ifx\csname graphtemp\endcsname\relax \csname newdimen\endcsname\graphtemp\fi
\setbox\graph=\vtop{\vskip 0pt\hbox{%
    \graphtemp=.5ex\advance\graphtemp by 0.628in
    \rlap{\kern 0.929in\lower\graphtemp\hbox to 0pt{\hss $\bu$\hss}}%
    \graphtemp=.5ex\advance\graphtemp by 0.628in
    \rlap{\kern 1.202in\lower\graphtemp\hbox to 0pt{\hss $\bu$\hss}}%
    \graphtemp=.5ex\advance\graphtemp by 0.901in
    \rlap{\kern 0.655in\lower\graphtemp\hbox to 0pt{\hss $\bu$\hss}}%
    \graphtemp=.5ex\advance\graphtemp by 0.901in
    \rlap{\kern 1.475in\lower\graphtemp\hbox to 0pt{\hss $\bu$\hss}}%
    \graphtemp=.5ex\advance\graphtemp by 0.355in
    \rlap{\kern 0.655in\lower\graphtemp\hbox to 0pt{\hss $\bu$\hss}}%
    \graphtemp=.5ex\advance\graphtemp by 0.355in
    \rlap{\kern 1.475in\lower\graphtemp\hbox to 0pt{\hss $\bu$\hss}}%
    \graphtemp=.5ex\advance\graphtemp by 0.901in
    \rlap{\kern 0.382in\lower\graphtemp\hbox to 0pt{\hss $\bu$\hss}}%
    \graphtemp=.5ex\advance\graphtemp by 0.901in
    \rlap{\kern 1.748in\lower\graphtemp\hbox to 0pt{\hss $\bu$\hss}}%
    \graphtemp=.5ex\advance\graphtemp by 0.355in
    \rlap{\kern 0.382in\lower\graphtemp\hbox to 0pt{\hss $\bu$\hss}}%
    \graphtemp=.5ex\advance\graphtemp by 0.355in
    \rlap{\kern 1.748in\lower\graphtemp\hbox to 0pt{\hss $\bu$\hss}}%
    \special{pn 11}%
    \special{pa 929 628}%
    \special{pa 1202 628}%
    \special{fp}%
    \special{pa 929 628}%
    \special{pa 655 901}%
    \special{fp}%
    \special{pa 929 628}%
    \special{pa 655 355}%
    \special{fp}%
    \special{pa 1202 628}%
    \special{pa 929 628}%
    \special{fp}%
    \special{pa 1202 628}%
    \special{pa 1475 901}%
    \special{fp}%
    \special{pa 1202 628}%
    \special{pa 1475 355}%
    \special{fp}%
    \special{pn 8}%
    \special{ar 1065 507 437 437 -2.786171 -0.355421}%
    \special{ar 1065 749 437 437 0.355421 2.786171}%
    \special{ar 470 628 287 287 1.880641 4.402544}%
    \special{ar 1660 628 287 287 -1.260952 1.260952}%
    \special{pn 11}%
    \special{pa 655 901}%
    \special{pa 382 901}%
    \special{fp}%
    \special{pa 1475 901}%
    \special{pa 1748 901}%
    \special{fp}%
    \special{pa 655 355}%
    \special{pa 382 355}%
    \special{fp}%
    \special{pa 1475 355}%
    \special{pa 1748 355}%
    \special{fp}%
    \graphtemp=.5ex\advance\graphtemp by 0.628in
    \rlap{\kern 0.819in\lower\graphtemp\hbox to 0pt{\hss $x$\hss}}%
    \graphtemp=.5ex\advance\graphtemp by 1.006in
    \rlap{\kern 0.578in\lower\graphtemp\hbox to 0pt{\hss $b$\hss}}%
    \graphtemp=.5ex\advance\graphtemp by 0.278in
    \rlap{\kern 0.578in\lower\graphtemp\hbox to 0pt{\hss $a$\hss}}%
    \graphtemp=.5ex\advance\graphtemp by 0.628in
    \rlap{\kern 1.311in\lower\graphtemp\hbox to 0pt{\hss $y$\hss}}%
    \graphtemp=.5ex\advance\graphtemp by 0.979in
    \rlap{\kern 1.552in\lower\graphtemp\hbox to 0pt{\hss $v$\hss}}%
    \graphtemp=.5ex\advance\graphtemp by 0.278in
    \rlap{\kern 1.552in\lower\graphtemp\hbox to 0pt{\hss $u$\hss}}%
    \graphtemp=.5ex\advance\graphtemp by 0.355in
    \rlap{\kern 1.065in\lower\graphtemp\hbox to 0pt{\hss $C$\hss}}%
    \graphtemp=.5ex\advance\graphtemp by 0.901in
    \rlap{\kern 1.065in\lower\graphtemp\hbox to 0pt{\hss $D$\hss}}%
    \graphtemp=.5ex\advance\graphtemp by 0.628in
    \rlap{\kern 0.000in\lower\graphtemp\hbox to 0pt{\hss $G$\hss}}%
    \graphtemp=.5ex\advance\graphtemp by 1.006in
    \rlap{\kern 0.305in\lower\graphtemp\hbox to 0pt{\hss $b'$\hss}}%
    \graphtemp=.5ex\advance\graphtemp by 0.278in
    \rlap{\kern 0.305in\lower\graphtemp\hbox to 0pt{\hss $a'$\hss}}%
    \graphtemp=.5ex\advance\graphtemp by 0.979in
    \rlap{\kern 1.825in\lower\graphtemp\hbox to 0pt{\hss $v'$\hss}}%
    \graphtemp=.5ex\advance\graphtemp by 0.628in
    \rlap{\kern 0.546in\lower\graphtemp\hbox to 0pt{\hss $Q$\hss}}%
    \graphtemp=.5ex\advance\graphtemp by 0.628in
    \rlap{\kern 1.584in\lower\graphtemp\hbox to 0pt{\hss $R$\hss}}%
    \graphtemp=.5ex\advance\graphtemp by 0.278in
    \rlap{\kern 1.825in\lower\graphtemp\hbox to 0pt{\hss $u'$\hss}}%
    \graphtemp=.5ex\advance\graphtemp by 0.109in
    \rlap{\kern 0.901in\lower\graphtemp\hbox to 0pt{\hss $\bu$\hss}}%
    \graphtemp=.5ex\advance\graphtemp by 0.032in
    \rlap{\kern 0.824in\lower\graphtemp\hbox to 0pt{\hss $z$\hss}}%
    \graphtemp=.5ex\advance\graphtemp by 0.628in
    \rlap{\kern 3.113in\lower\graphtemp\hbox to 0pt{\hss $\bu$\hss}}%
    \graphtemp=.5ex\advance\graphtemp by 0.628in
    \rlap{\kern 3.387in\lower\graphtemp\hbox to 0pt{\hss $\bu$\hss}}%
    \graphtemp=.5ex\advance\graphtemp by 0.901in
    \rlap{\kern 2.840in\lower\graphtemp\hbox to 0pt{\hss $\bu$\hss}}%
    \graphtemp=.5ex\advance\graphtemp by 0.901in
    \rlap{\kern 3.660in\lower\graphtemp\hbox to 0pt{\hss $\bu$\hss}}%
    \graphtemp=.5ex\advance\graphtemp by 0.355in
    \rlap{\kern 2.840in\lower\graphtemp\hbox to 0pt{\hss $\bu$\hss}}%
    \graphtemp=.5ex\advance\graphtemp by 0.355in
    \rlap{\kern 3.660in\lower\graphtemp\hbox to 0pt{\hss $\bu$\hss}}%
    \graphtemp=.5ex\advance\graphtemp by 0.901in
    \rlap{\kern 2.567in\lower\graphtemp\hbox to 0pt{\hss $\bu$\hss}}%
    \graphtemp=.5ex\advance\graphtemp by 0.901in
    \rlap{\kern 3.933in\lower\graphtemp\hbox to 0pt{\hss $\bu$\hss}}%
    \graphtemp=.5ex\advance\graphtemp by 0.355in
    \rlap{\kern 2.567in\lower\graphtemp\hbox to 0pt{\hss $\bu$\hss}}%
    \graphtemp=.5ex\advance\graphtemp by 0.355in
    \rlap{\kern 3.933in\lower\graphtemp\hbox to 0pt{\hss $\bu$\hss}}%
    \special{pa 3113 628}%
    \special{pa 3387 628}%
    \special{fp}%
    \special{pa 3113 628}%
    \special{pa 3660 901}%
    \special{fp}%
    \special{pa 3113 628}%
    \special{pa 2840 355}%
    \special{fp}%
    \special{pa 3387 628}%
    \special{pa 3113 628}%
    \special{fp}%
    \special{pa 3387 628}%
    \special{pa 2840 901}%
    \special{fp}%
    \special{pa 3387 628}%
    \special{pa 3660 355}%
    \special{fp}%
    \special{pn 8}%
    \special{ar 3250 507 437 437 -2.786171 -0.355421}%
    \special{ar 3250 749 437 437 0.355421 2.786171}%
    \special{ar 2655 628 287 287 1.880641 4.402544}%
    \special{ar 3845 628 287 287 -1.260952 1.260952}%
    \special{pn 11}%
    \special{pa 2840 901}%
    \special{pa 2567 901}%
    \special{fp}%
    \special{pa 3660 901}%
    \special{pa 3933 901}%
    \special{fp}%
    \special{pa 2840 355}%
    \special{pa 2567 355}%
    \special{fp}%
    \special{pa 3660 355}%
    \special{pa 3933 355}%
    \special{fp}%
    \graphtemp=.5ex\advance\graphtemp by 0.683in
    \rlap{\kern 3.004in\lower\graphtemp\hbox to 0pt{\hss $x'$\hss}}%
    \graphtemp=.5ex\advance\graphtemp by 1.006in
    \rlap{\kern 2.763in\lower\graphtemp\hbox to 0pt{\hss $b$\hss}}%
    \graphtemp=.5ex\advance\graphtemp by 0.278in
    \rlap{\kern 2.763in\lower\graphtemp\hbox to 0pt{\hss $a$\hss}}%
    \graphtemp=.5ex\advance\graphtemp by 0.628in
    \rlap{\kern 3.496in\lower\graphtemp\hbox to 0pt{\hss $y'$\hss}}%
    \graphtemp=.5ex\advance\graphtemp by 0.979in
    \rlap{\kern 3.737in\lower\graphtemp\hbox to 0pt{\hss $v$\hss}}%
    \graphtemp=.5ex\advance\graphtemp by 0.278in
    \rlap{\kern 3.737in\lower\graphtemp\hbox to 0pt{\hss $u$\hss}}%
    \graphtemp=.5ex\advance\graphtemp by 0.355in
    \rlap{\kern 3.250in\lower\graphtemp\hbox to 0pt{\hss $C'$\hss}}%
    \graphtemp=.5ex\advance\graphtemp by 0.929in
    \rlap{\kern 3.250in\lower\graphtemp\hbox to 0pt{\hss $D'$\hss}}%
    \graphtemp=.5ex\advance\graphtemp by 0.628in
    \rlap{\kern 4.315in\lower\graphtemp\hbox to 0pt{\hss $H$\hss}}%
    \graphtemp=.5ex\advance\graphtemp by 1.006in
    \rlap{\kern 2.490in\lower\graphtemp\hbox to 0pt{\hss $b'$\hss}}%
    \graphtemp=.5ex\advance\graphtemp by 0.278in
    \rlap{\kern 2.490in\lower\graphtemp\hbox to 0pt{\hss $a'$\hss}}%
    \graphtemp=.5ex\advance\graphtemp by 0.979in
    \rlap{\kern 4.010in\lower\graphtemp\hbox to 0pt{\hss $v'$\hss}}%
    \graphtemp=.5ex\advance\graphtemp by 0.628in
    \rlap{\kern 2.731in\lower\graphtemp\hbox to 0pt{\hss $~$\hss}}%
    \graphtemp=.5ex\advance\graphtemp by 0.628in
    \rlap{\kern 3.769in\lower\graphtemp\hbox to 0pt{\hss $~$\hss}}%
    \graphtemp=.5ex\advance\graphtemp by 0.278in
    \rlap{\kern 4.010in\lower\graphtemp\hbox to 0pt{\hss $u'$\hss}}%
    \graphtemp=.5ex\advance\graphtemp by 0.109in
    \rlap{\kern 3.086in\lower\graphtemp\hbox to 0pt{\hss $\bu$\hss}}%
    \graphtemp=.5ex\advance\graphtemp by 0.032in
    \rlap{\kern 3.009in\lower\graphtemp\hbox to 0pt{\hss $z$\hss}}%
    \graphtemp=.5ex\advance\graphtemp by 0.628in
    \rlap{\kern 5.298in\lower\graphtemp\hbox to 0pt{\hss $\bu$\hss}}%
    \graphtemp=.5ex\advance\graphtemp by 0.628in
    \rlap{\kern 5.571in\lower\graphtemp\hbox to 0pt{\hss $\bu$\hss}}%
    \graphtemp=.5ex\advance\graphtemp by 0.901in
    \rlap{\kern 5.025in\lower\graphtemp\hbox to 0pt{\hss $\bu$\hss}}%
    \graphtemp=.5ex\advance\graphtemp by 0.901in
    \rlap{\kern 5.845in\lower\graphtemp\hbox to 0pt{\hss $\bu$\hss}}%
    \graphtemp=.5ex\advance\graphtemp by 0.355in
    \rlap{\kern 5.025in\lower\graphtemp\hbox to 0pt{\hss $\bu$\hss}}%
    \graphtemp=.5ex\advance\graphtemp by 0.355in
    \rlap{\kern 5.845in\lower\graphtemp\hbox to 0pt{\hss $\bu$\hss}}%
    \graphtemp=.5ex\advance\graphtemp by 0.901in
    \rlap{\kern 4.752in\lower\graphtemp\hbox to 0pt{\hss $\bu$\hss}}%
    \graphtemp=.5ex\advance\graphtemp by 0.901in
    \rlap{\kern 6.118in\lower\graphtemp\hbox to 0pt{\hss $\bu$\hss}}%
    \graphtemp=.5ex\advance\graphtemp by 0.355in
    \rlap{\kern 4.752in\lower\graphtemp\hbox to 0pt{\hss $\bu$\hss}}%
    \graphtemp=.5ex\advance\graphtemp by 0.355in
    \rlap{\kern 6.118in\lower\graphtemp\hbox to 0pt{\hss $\bu$\hss}}%
    \special{pa 5298 628}%
    \special{pa 5571 628}%
    \special{fp}%
    \special{pa 5298 628}%
    \special{pa 4752 355}%
    \special{fp}%
    \special{pa 5298 628}%
    \special{pa 5025 355}%
    \special{fp}%
    \special{pa 5571 628}%
    \special{pa 5298 628}%
    \special{fp}%
    \special{pa 5571 628}%
    \special{pa 5845 901}%
    \special{fp}%
    \special{pa 5571 628}%
    \special{pa 5845 355}%
    \special{fp}%
    \special{pn 8}%
    \special{ar 5435 507 437 437 -2.786171 -0.355421}%
    \special{ar 5435 749 437 437 0.355421 2.786171}%
    \special{ar 4840 628 287 287 1.880641 4.402544}%
    \special{ar 6030 628 287 287 -1.260952 1.260952}%
    \special{pn 11}%
    \special{pa 5025 901}%
    \special{pa 4752 901}%
    \special{fp}%
    \special{pa 5845 901}%
    \special{pa 6118 901}%
    \special{fp}%
    \special{pa 5025 355}%
    \special{pa 5025 901}%
    \special{fp}%
    \special{pa 5845 355}%
    \special{pa 6118 355}%
    \special{fp}%
    \graphtemp=.5ex\advance\graphtemp by 0.737in
    \rlap{\kern 5.298in\lower\graphtemp\hbox to 0pt{\hss $x''$\hss}}%
    \graphtemp=.5ex\advance\graphtemp by 1.006in
    \rlap{\kern 4.948in\lower\graphtemp\hbox to 0pt{\hss $b$\hss}}%
    \graphtemp=.5ex\advance\graphtemp by 0.278in
    \rlap{\kern 4.948in\lower\graphtemp\hbox to 0pt{\hss $a''$\hss}}%
    \graphtemp=.5ex\advance\graphtemp by 0.628in
    \rlap{\kern 5.681in\lower\graphtemp\hbox to 0pt{\hss $y$\hss}}%
    \graphtemp=.5ex\advance\graphtemp by 0.979in
    \rlap{\kern 5.922in\lower\graphtemp\hbox to 0pt{\hss $v$\hss}}%
    \graphtemp=.5ex\advance\graphtemp by 0.278in
    \rlap{\kern 5.922in\lower\graphtemp\hbox to 0pt{\hss $u$\hss}}%
    \graphtemp=.5ex\advance\graphtemp by 0.355in
    \rlap{\kern 5.435in\lower\graphtemp\hbox to 0pt{\hss $C''$\hss}}%
    \graphtemp=.5ex\advance\graphtemp by 0.901in
    \rlap{\kern 5.435in\lower\graphtemp\hbox to 0pt{\hss $~$\hss}}%
    \graphtemp=.5ex\advance\graphtemp by 0.628in
    \rlap{\kern 6.500in\lower\graphtemp\hbox to 0pt{\hss $H'$\hss}}%
    \graphtemp=.5ex\advance\graphtemp by 1.006in
    \rlap{\kern 4.675in\lower\graphtemp\hbox to 0pt{\hss $b'$\hss}}%
    \graphtemp=.5ex\advance\graphtemp by 0.278in
    \rlap{\kern 4.675in\lower\graphtemp\hbox to 0pt{\hss $a'$\hss}}%
    \graphtemp=.5ex\advance\graphtemp by 0.979in
    \rlap{\kern 6.195in\lower\graphtemp\hbox to 0pt{\hss $v'$\hss}}%
    \graphtemp=.5ex\advance\graphtemp by 0.628in
    \rlap{\kern 4.916in\lower\graphtemp\hbox to 0pt{\hss $Q$\hss}}%
    \graphtemp=.5ex\advance\graphtemp by 0.628in
    \rlap{\kern 5.954in\lower\graphtemp\hbox to 0pt{\hss $R$\hss}}%
    \graphtemp=.5ex\advance\graphtemp by 0.278in
    \rlap{\kern 6.195in\lower\graphtemp\hbox to 0pt{\hss $u'$\hss}}%
    \graphtemp=.5ex\advance\graphtemp by 0.109in
    \rlap{\kern 5.271in\lower\graphtemp\hbox to 0pt{\hss $\bu$\hss}}%
    \graphtemp=.5ex\advance\graphtemp by 0.032in
    \rlap{\kern 5.194in\lower\graphtemp\hbox to 0pt{\hss $z$\hss}}%
    \hbox{\vrule depth1.186in width0pt height 0pt}%
    \kern 6.500in
  }%
}%
}
\caption{Many $g$-cycles through $\{x,y\}$.}\label{manyg}
\end{figure}

\eject

\begin{lemma}\label{step2}
No vertex lies in three $g$-cycles.
\end{lemma}
\begin{proof}
Suppose that $y$ with neighborhood $\{x,u,v\}$ lies in three $g$-cycles in $G$.
Each of these $g$-cycles uses two edges at $y$; any two of them have one
common edge.  With $a,b\in N_G(x)$, $c,d\in N_G(u)$, and $e,f\in N_G(v)$,
label the vertices so that the three cycles $C$, $D$, and $R$ contain the paths
$\la a,x,y,u,d\ra$, $\la c,u,y,v,f\ra$, and $\la e,v,y,x,b\ra$, respectively
(see Figure~\ref{threeg}).  Since $g\ge5$, these 10 vertices are distinct.

Since $xy$ lies in the $g$-cycle $C$ containing $\la a,x,y,u,d\ra$,
the card $G-\{x,y\}$ has only one alternative reconstruction $H$.
As in Remark~\ref{alt}, we may obtain $H$ from the card by adding
$x'$ and $y'$ with $N_H(x')=\{y',a,v\}$ and $N_H(y')=\{x',u,b\}$.

By Lemma~\ref{step1}, $G$ has no $g$-cycle through $\la a,x,b\ra$.
Hence $G$ has exactly one $g$-cycle containing exactly one of $\{x,y\}$.
By Lemma~\ref{shortest}, in $H$ exactly one $g$-cycle $Q$ contains exactly one
of $\{x',y'\}$.  Hence $Q$ contains $\la a,x',v\ra$ or $\la b,y',u\ra$.
Avoiding spliced $g$-cycles in $H$ implies that $Q$ contains $vf$ in the 
first case and $uc$ in the second case.  Hence $G$ contains an $a,f$-path
or a $b,c$-path of length $g-3$, and not both.

Applying the symmetric argument to $G-\{u,y\}$ and $G-\{v,y\}$ yields paths
with length $g-3$ in $G$ whose endpoints are exactly one pair in each of the
following three sets: $\{(a,f),(b,c)\}$, $\{(c,b),(d,e)\}$, $\{(f,a),(e,d)\}$.
This is impossible: as soon as one pair of endpoints is picked, it satisfies
one other set, which then prevents the third set from contributing a pair.
\end{proof}

\vspace{-2pc}
\begin{figure}[hbt]
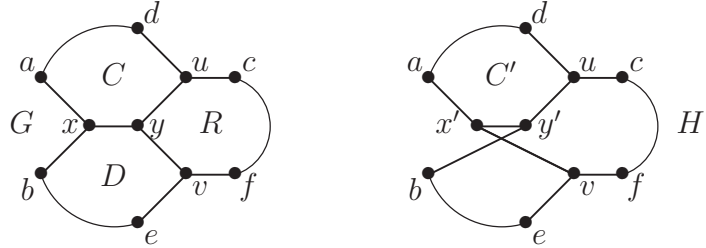

\gpic{
\expandafter\ifx\csname graph\endcsname\relax \csname newbox\endcsname\graph\fi
\expandafter\ifx\csname graphtemp\endcsname\relax \csname newdimen\endcsname\graphtemp\fi
\setbox\graph=\vtop{\vskip 0pt\hbox{%
    \graphtemp=.5ex\advance\graphtemp by 0.609in
    \rlap{\kern 0.355in\lower\graphtemp\hbox to 0pt{\hss $\bu$\hss}}%
    \graphtemp=.5ex\advance\graphtemp by 0.609in
    \rlap{\kern 0.609in\lower\graphtemp\hbox to 0pt{\hss $\bu$\hss}}%
    \graphtemp=.5ex\advance\graphtemp by 0.862in
    \rlap{\kern 0.101in\lower\graphtemp\hbox to 0pt{\hss $\bu$\hss}}%
    \graphtemp=.5ex\advance\graphtemp by 0.862in
    \rlap{\kern 0.862in\lower\graphtemp\hbox to 0pt{\hss $\bu$\hss}}%
    \graphtemp=.5ex\advance\graphtemp by 0.355in
    \rlap{\kern 0.101in\lower\graphtemp\hbox to 0pt{\hss $\bu$\hss}}%
    \graphtemp=.5ex\advance\graphtemp by 0.355in
    \rlap{\kern 0.862in\lower\graphtemp\hbox to 0pt{\hss $\bu$\hss}}%
    \graphtemp=.5ex\advance\graphtemp by 0.101in
    \rlap{\kern 0.609in\lower\graphtemp\hbox to 0pt{\hss $\bu$\hss}}%
    \graphtemp=.5ex\advance\graphtemp by 0.862in
    \rlap{\kern 1.116in\lower\graphtemp\hbox to 0pt{\hss $\bu$\hss}}%
    \graphtemp=.5ex\advance\graphtemp by 0.355in
    \rlap{\kern 1.116in\lower\graphtemp\hbox to 0pt{\hss $\bu$\hss}}%
    \graphtemp=.5ex\advance\graphtemp by 1.116in
    \rlap{\kern 0.609in\lower\graphtemp\hbox to 0pt{\hss $\bu$\hss}}%
    \special{pn 11}%
    \special{pa 355 609}%
    \special{pa 609 609}%
    \special{fp}%
    \special{pa 355 609}%
    \special{pa 101 862}%
    \special{fp}%
    \special{pa 355 609}%
    \special{pa 101 355}%
    \special{fp}%
    \special{pa 862 862}%
    \special{pa 609 609}%
    \special{fp}%
    \special{pa 862 862}%
    \special{pa 609 1116}%
    \special{fp}%
    \special{pa 862 862}%
    \special{pa 1116 862}%
    \special{fp}%
    \special{pa 862 355}%
    \special{pa 609 609}%
    \special{fp}%
    \special{pa 862 355}%
    \special{pa 609 101}%
    \special{fp}%
    \special{pa 862 355}%
    \special{pa 1116 355}%
    \special{fp}%
    \special{pn 8}%
    \special{ar 485 488 406 406 -2.808122 -1.260766}%
    \special{ar 485 729 406 406 1.260766 2.808122}%
    \special{ar 1035 609 266 266 -1.260952 1.260952}%
    \graphtemp=.5ex\advance\graphtemp by 0.609in
    \rlap{\kern 0.254in\lower\graphtemp\hbox to 0pt{\hss $x$\hss}}%
    \graphtemp=.5ex\advance\graphtemp by 0.959in
    \rlap{\kern 0.030in\lower\graphtemp\hbox to 0pt{\hss $b$\hss}}%
    \graphtemp=.5ex\advance\graphtemp by 0.283in
    \rlap{\kern 0.030in\lower\graphtemp\hbox to 0pt{\hss $a$\hss}}%
    \graphtemp=.5ex\advance\graphtemp by 0.609in
    \rlap{\kern 0.710in\lower\graphtemp\hbox to 0pt{\hss $y$\hss}}%
    \graphtemp=.5ex\advance\graphtemp by 0.934in
    \rlap{\kern 0.934in\lower\graphtemp\hbox to 0pt{\hss $v$\hss}}%
    \graphtemp=.5ex\advance\graphtemp by 0.283in
    \rlap{\kern 0.934in\lower\graphtemp\hbox to 0pt{\hss $u$\hss}}%
    \graphtemp=.5ex\advance\graphtemp by 0.030in
    \rlap{\kern 0.680in\lower\graphtemp\hbox to 0pt{\hss $d$\hss}}%
    \graphtemp=.5ex\advance\graphtemp by 0.283in
    \rlap{\kern 1.188in\lower\graphtemp\hbox to 0pt{\hss $c$\hss}}%
    \graphtemp=.5ex\advance\graphtemp by 0.934in
    \rlap{\kern 1.188in\lower\graphtemp\hbox to 0pt{\hss $f$\hss}}%
    \graphtemp=.5ex\advance\graphtemp by 0.609in
    \rlap{\kern 0.000in\lower\graphtemp\hbox to 0pt{\hss $G$\hss}}%
    \graphtemp=.5ex\advance\graphtemp by 0.609in
    \rlap{\kern 0.964in\lower\graphtemp\hbox to 0pt{\hss $~$\hss}}%
    \graphtemp=.5ex\advance\graphtemp by 1.188in
    \rlap{\kern 0.680in\lower\graphtemp\hbox to 0pt{\hss $e$\hss}}%
    \graphtemp=.5ex\advance\graphtemp by 0.888in
    \rlap{\kern 0.482in\lower\graphtemp\hbox to 0pt{\hss $D$\hss}}%
    \graphtemp=.5ex\advance\graphtemp by 0.355in
    \rlap{\kern 0.482in\lower\graphtemp\hbox to 0pt{\hss $C$\hss}}%
    \graphtemp=.5ex\advance\graphtemp by 0.609in
    \rlap{\kern 0.989in\lower\graphtemp\hbox to 0pt{\hss $R$\hss}}%
    \graphtemp=.5ex\advance\graphtemp by 0.609in
    \rlap{\kern 2.384in\lower\graphtemp\hbox to 0pt{\hss $\bu$\hss}}%
    \graphtemp=.5ex\advance\graphtemp by 0.609in
    \rlap{\kern 2.638in\lower\graphtemp\hbox to 0pt{\hss $\bu$\hss}}%
    \graphtemp=.5ex\advance\graphtemp by 0.862in
    \rlap{\kern 2.130in\lower\graphtemp\hbox to 0pt{\hss $\bu$\hss}}%
    \graphtemp=.5ex\advance\graphtemp by 0.862in
    \rlap{\kern 2.891in\lower\graphtemp\hbox to 0pt{\hss $\bu$\hss}}%
    \graphtemp=.5ex\advance\graphtemp by 0.355in
    \rlap{\kern 2.130in\lower\graphtemp\hbox to 0pt{\hss $\bu$\hss}}%
    \graphtemp=.5ex\advance\graphtemp by 0.355in
    \rlap{\kern 2.891in\lower\graphtemp\hbox to 0pt{\hss $\bu$\hss}}%
    \graphtemp=.5ex\advance\graphtemp by 0.101in
    \rlap{\kern 2.638in\lower\graphtemp\hbox to 0pt{\hss $\bu$\hss}}%
    \graphtemp=.5ex\advance\graphtemp by 0.862in
    \rlap{\kern 3.145in\lower\graphtemp\hbox to 0pt{\hss $\bu$\hss}}%
    \graphtemp=.5ex\advance\graphtemp by 0.355in
    \rlap{\kern 3.145in\lower\graphtemp\hbox to 0pt{\hss $\bu$\hss}}%
    \graphtemp=.5ex\advance\graphtemp by 1.116in
    \rlap{\kern 2.638in\lower\graphtemp\hbox to 0pt{\hss $\bu$\hss}}%
    \special{pn 11}%
    \special{pa 2384 609}%
    \special{pa 2638 609}%
    \special{fp}%
    \special{pa 2384 609}%
    \special{pa 2891 862}%
    \special{fp}%
    \special{pa 2384 609}%
    \special{pa 2130 355}%
    \special{fp}%
    \special{pa 2891 862}%
    \special{pa 2384 609}%
    \special{fp}%
    \special{pa 2891 862}%
    \special{pa 2638 1116}%
    \special{fp}%
    \special{pa 2891 862}%
    \special{pa 3145 862}%
    \special{fp}%
    \special{pa 2891 355}%
    \special{pa 2638 609}%
    \special{fp}%
    \special{pa 2891 355}%
    \special{pa 2638 101}%
    \special{fp}%
    \special{pa 2891 355}%
    \special{pa 3145 355}%
    \special{fp}%
    \special{pa 2130 862}%
    \special{pa 2638 609}%
    \special{fp}%
    \special{pn 8}%
    \special{ar 2514 488 406 406 -2.808122 -1.260766}%
    \special{ar 2514 729 406 406 1.260766 2.808122}%
    \special{ar 3064 609 266 266 -1.260952 1.260952}%
    \graphtemp=.5ex\advance\graphtemp by 0.609in
    \rlap{\kern 2.232in\lower\graphtemp\hbox to 0pt{\hss $x'$\hss}}%
    \graphtemp=.5ex\advance\graphtemp by 0.959in
    \rlap{\kern 2.059in\lower\graphtemp\hbox to 0pt{\hss $b$\hss}}%
    \graphtemp=.5ex\advance\graphtemp by 0.283in
    \rlap{\kern 2.059in\lower\graphtemp\hbox to 0pt{\hss $a$\hss}}%
    \graphtemp=.5ex\advance\graphtemp by 0.609in
    \rlap{\kern 2.764in\lower\graphtemp\hbox to 0pt{\hss $y'$\hss}}%
    \graphtemp=.5ex\advance\graphtemp by 0.934in
    \rlap{\kern 2.963in\lower\graphtemp\hbox to 0pt{\hss $v$\hss}}%
    \graphtemp=.5ex\advance\graphtemp by 0.283in
    \rlap{\kern 2.963in\lower\graphtemp\hbox to 0pt{\hss $u$\hss}}%
    \graphtemp=.5ex\advance\graphtemp by 0.030in
    \rlap{\kern 2.709in\lower\graphtemp\hbox to 0pt{\hss $d$\hss}}%
    \graphtemp=.5ex\advance\graphtemp by 0.283in
    \rlap{\kern 3.217in\lower\graphtemp\hbox to 0pt{\hss $c$\hss}}%
    \graphtemp=.5ex\advance\graphtemp by 0.934in
    \rlap{\kern 3.217in\lower\graphtemp\hbox to 0pt{\hss $f$\hss}}%
    \graphtemp=.5ex\advance\graphtemp by 0.609in
    \rlap{\kern 3.500in\lower\graphtemp\hbox to 0pt{\hss $H$\hss}}%
    \graphtemp=.5ex\advance\graphtemp by 0.355in
    \rlap{\kern 2.511in\lower\graphtemp\hbox to 0pt{\hss $C'$\hss}}%
    \graphtemp=.5ex\advance\graphtemp by 1.188in
    \rlap{\kern 2.709in\lower\graphtemp\hbox to 0pt{\hss $e$\hss}}%
    \hbox{\vrule depth1.217in width0pt height 0pt}%
    \kern 3.500in
  }%
}%
}
\caption{A vertex in three $g$-cycles.}\label{threeg}
\end{figure}

\vspace{-1pc}
\begin{lemma}\label{step3}
No vertex lies in two $g$-cycles.
\end{lemma}
\begin{proof}
Since $G$ is $3$-regular, a vertex in $g$-cycles $C$ and $D$ requires an edge
$xy$ in $C$ and $D$ (only one common edge, since there are no spliced
$g$-cycles).  Label vertices as in Remark~\ref{alt}, with $\la z,a,x,y,u\ra$
lying along a $g$-cycle $C$.  Since $G$ has girth $g$, the neighbor $w$ of $a$
that is not on $C$ is not on the other $g$-cycle through $xy$.  Note that $ax$
is not in another $g$-cycle, since that would put $x$ on three $g$-cycles.

Since $ax$ lies in a $g$-cycle, $G-\{a,x\}$ has only one alternative
reconstruction, $H$.  We may label $H$ so the $g$-cycle $C'$ through the
missing vertices $a'$ and $x'$ arises from $C$ by replacing $a$ with $a'$ and
$x$ with $x'$, and so $N_H(a')=\{x',z,b\}$ and $N_H(x')=\{a',y,w\}$
(see Figure~\ref{twog}).

\nobreak
Since $G$ has a $g$-cycle $D$ through exactly one of $\{a,x\}$, also $H$ has a
$g$-cycle $D'$ through exactly one of $\{a',x'\}$.  This cycle cannot use
$a'x'$, so it uses $\la w,x',y\ra$ or $\la b,a',z\ra$.  In each case, we will
obtain an alternative reconstruction from the deck that has three $g$-cycles
containing a single vertex, which by Lemma~\ref{step2} is forbidden.

In the first case, $D'$ cannot use $yu$, since $H$ has no spliced $g$-cycles.
Hence $D'$ uses $yv$, and there is a $v,w$-path $P$ of length $g-3$ in $H$
and $G$ (not using $y$).  Note that $P$ completes a $g'$-cycle $Q$ in $G$ with
$\la w,a,x,y,v\ra$.  Since $xy$ lies on a $g$-cycle, $G-\{x,y\}$ has only
one alternative reconstruction, $H'$ (see Figure~\ref{twog}).  We may label the
missing vertices $x''$ and $y''$ so that $ax'',y''u\in E(H')$, which forces
$vx'',y''b\in E(H')$.  Now replacing $\la w,a,x,y,v\ra$ in $Q$ with
$\la w,a,x'',v\ra$ yields three $g$-cycles in $H'$ containing $x''$.

In the second case, $D'$ must continue after $\la b,a',z\ra$ to
the neighbor $z'$ of $z$ not on $C'$, since $H$ has no spliced $g$-cycles.
Replacing $a'b$ in $D'$ with $\la a,x,b\ra$ yields a $g'$-cycle $R$ in $G$
containing the path $\la z',z,a,x,b\ra$.  Since $az$ lies on a $g$-cycle, there
is a unique alternative reconstruction from $G-\{a,z\}$; call it $H''$.  We may
label $H''$ so its missing vertices $a''$ and $z''$ replace $a$ and $z$ in $C$
to form a $g$-cycle $C''$, and then the remaining edges incident to
$\{a'',z''\}$ must be $z'a''$ and $wz''$, as in Figure~\ref{twog}.
Now replacing $\la z',z,a,x,b\ra$ in $R$ with $\la z',a'',x,b\ra$ yields
three $g$-cycles in $H''$ containing $x$.
\end{proof}

\vspace{-1pc}
\begin{figure}[hbt]
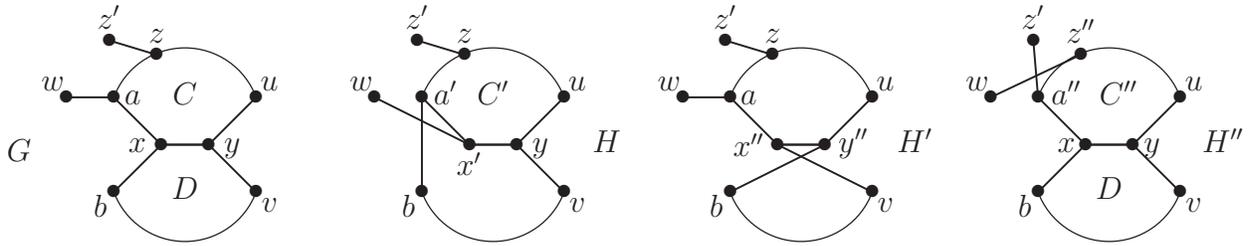

\gpic{
\expandafter\ifx\csname graph\endcsname\relax \csname newbox\endcsname\graph\fi
\expandafter\ifx\csname graphtemp\endcsname\relax \csname newdimen\endcsname\graphtemp\fi
\setbox\graph=\vtop{\vskip 0pt\hbox{%
    \graphtemp=.5ex\advance\graphtemp by 0.645in
    \rlap{\kern 0.844in\lower\graphtemp\hbox to 0pt{\hss $\bu$\hss}}%
    \graphtemp=.5ex\advance\graphtemp by 0.645in
    \rlap{\kern 1.092in\lower\graphtemp\hbox to 0pt{\hss $\bu$\hss}}%
    \graphtemp=.5ex\advance\graphtemp by 0.893in
    \rlap{\kern 0.595in\lower\graphtemp\hbox to 0pt{\hss $\bu$\hss}}%
    \graphtemp=.5ex\advance\graphtemp by 0.893in
    \rlap{\kern 1.340in\lower\graphtemp\hbox to 0pt{\hss $\bu$\hss}}%
    \graphtemp=.5ex\advance\graphtemp by 0.397in
    \rlap{\kern 0.595in\lower\graphtemp\hbox to 0pt{\hss $\bu$\hss}}%
    \graphtemp=.5ex\advance\graphtemp by 0.397in
    \rlap{\kern 1.340in\lower\graphtemp\hbox to 0pt{\hss $\bu$\hss}}%
    \graphtemp=.5ex\advance\graphtemp by 0.397in
    \rlap{\kern 0.347in\lower\graphtemp\hbox to 0pt{\hss $\bu$\hss}}%
    \graphtemp=.5ex\advance\graphtemp by 0.099in
    \rlap{\kern 0.571in\lower\graphtemp\hbox to 0pt{\hss $\bu$\hss}}%
    \special{pn 11}%
    \special{pa 844 645}%
    \special{pa 1092 645}%
    \special{fp}%
    \special{pa 844 645}%
    \special{pa 595 893}%
    \special{fp}%
    \special{pa 844 645}%
    \special{pa 595 397}%
    \special{fp}%
    \special{pa 1092 645}%
    \special{pa 844 645}%
    \special{fp}%
    \special{pa 1092 645}%
    \special{pa 1340 893}%
    \special{fp}%
    \special{pa 1092 645}%
    \special{pa 1340 397}%
    \special{fp}%
    \special{pn 8}%
    \special{ar 968 535 397 397 -2.786171 -0.355421}%
    \special{ar 968 755 397 397 0.355421 2.786171}%
    \special{pn 11}%
    \special{pa 595 397}%
    \special{pa 347 397}%
    \special{fp}%
    \graphtemp=.5ex\advance\graphtemp by 0.645in
    \rlap{\kern 0.719in\lower\graphtemp\hbox to 0pt{\hss $x$\hss}}%
    \graphtemp=.5ex\advance\graphtemp by 0.988in
    \rlap{\kern 0.525in\lower\graphtemp\hbox to 0pt{\hss $b$\hss}}%
    \graphtemp=.5ex\advance\graphtemp by 0.397in
    \rlap{\kern 0.695in\lower\graphtemp\hbox to 0pt{\hss $a$\hss}}%
    \graphtemp=.5ex\advance\graphtemp by 0.645in
    \rlap{\kern 1.216in\lower\graphtemp\hbox to 0pt{\hss $y$\hss}}%
    \graphtemp=.5ex\advance\graphtemp by 0.963in
    \rlap{\kern 1.410in\lower\graphtemp\hbox to 0pt{\hss $v$\hss}}%
    \graphtemp=.5ex\advance\graphtemp by 0.327in
    \rlap{\kern 1.410in\lower\graphtemp\hbox to 0pt{\hss $u$\hss}}%
    \graphtemp=.5ex\advance\graphtemp by 0.397in
    \rlap{\kern 0.968in\lower\graphtemp\hbox to 0pt{\hss $C$\hss}}%
    \graphtemp=.5ex\advance\graphtemp by 0.893in
    \rlap{\kern 0.968in\lower\graphtemp\hbox to 0pt{\hss $D$\hss}}%
    \graphtemp=.5ex\advance\graphtemp by 0.695in
    \rlap{\kern 0.099in\lower\graphtemp\hbox to 0pt{\hss $G$\hss}}%
    \graphtemp=.5ex\advance\graphtemp by 0.174in
    \rlap{\kern 0.819in\lower\graphtemp\hbox to 0pt{\hss $\bu$\hss}}%
    \graphtemp=.5ex\advance\graphtemp by 0.000in
    \rlap{\kern 0.571in\lower\graphtemp\hbox to 0pt{\hss $z'$\hss}}%
    \graphtemp=.5ex\advance\graphtemp by 0.074in
    \rlap{\kern 0.819in\lower\graphtemp\hbox to 0pt{\hss $z$\hss}}%
    \graphtemp=.5ex\advance\graphtemp by 0.327in
    \rlap{\kern 0.277in\lower\graphtemp\hbox to 0pt{\hss $w$\hss}}%
    \special{pa 571 99}%
    \special{pa 819 174}%
    \special{fp}%
    \graphtemp=.5ex\advance\graphtemp by 0.645in
    \rlap{\kern 2.456in\lower\graphtemp\hbox to 0pt{\hss $\bu$\hss}}%
    \graphtemp=.5ex\advance\graphtemp by 0.645in
    \rlap{\kern 2.704in\lower\graphtemp\hbox to 0pt{\hss $\bu$\hss}}%
    \graphtemp=.5ex\advance\graphtemp by 0.893in
    \rlap{\kern 2.208in\lower\graphtemp\hbox to 0pt{\hss $\bu$\hss}}%
    \graphtemp=.5ex\advance\graphtemp by 0.893in
    \rlap{\kern 2.952in\lower\graphtemp\hbox to 0pt{\hss $\bu$\hss}}%
    \graphtemp=.5ex\advance\graphtemp by 0.397in
    \rlap{\kern 2.208in\lower\graphtemp\hbox to 0pt{\hss $\bu$\hss}}%
    \graphtemp=.5ex\advance\graphtemp by 0.397in
    \rlap{\kern 2.952in\lower\graphtemp\hbox to 0pt{\hss $\bu$\hss}}%
    \graphtemp=.5ex\advance\graphtemp by 0.397in
    \rlap{\kern 1.960in\lower\graphtemp\hbox to 0pt{\hss $\bu$\hss}}%
    \graphtemp=.5ex\advance\graphtemp by 0.099in
    \rlap{\kern 2.183in\lower\graphtemp\hbox to 0pt{\hss $\bu$\hss}}%
    \special{pa 2456 645}%
    \special{pa 2704 645}%
    \special{fp}%
    \special{pa 2456 645}%
    \special{pa 1960 397}%
    \special{fp}%
    \special{pa 2456 645}%
    \special{pa 2208 397}%
    \special{fp}%
    \special{pa 2704 645}%
    \special{pa 2456 645}%
    \special{fp}%
    \special{pa 2704 645}%
    \special{pa 2952 893}%
    \special{fp}%
    \special{pa 2704 645}%
    \special{pa 2952 397}%
    \special{fp}%
    \special{pn 8}%
    \special{ar 2580 535 397 397 -2.786171 -0.355421}%
    \special{ar 2580 755 397 397 0.355421 2.786171}%
    \special{pn 11}%
    \special{pa 2208 397}%
    \special{pa 2208 893}%
    \special{fp}%
    \graphtemp=.5ex\advance\graphtemp by 0.769in
    \rlap{\kern 2.456in\lower\graphtemp\hbox to 0pt{\hss $x'$\hss}}%
    \graphtemp=.5ex\advance\graphtemp by 0.988in
    \rlap{\kern 2.138in\lower\graphtemp\hbox to 0pt{\hss $b$\hss}}%
    \graphtemp=.5ex\advance\graphtemp by 0.397in
    \rlap{\kern 2.332in\lower\graphtemp\hbox to 0pt{\hss $a'$\hss}}%
    \graphtemp=.5ex\advance\graphtemp by 0.645in
    \rlap{\kern 2.828in\lower\graphtemp\hbox to 0pt{\hss $y$\hss}}%
    \graphtemp=.5ex\advance\graphtemp by 0.963in
    \rlap{\kern 3.022in\lower\graphtemp\hbox to 0pt{\hss $v$\hss}}%
    \graphtemp=.5ex\advance\graphtemp by 0.327in
    \rlap{\kern 3.022in\lower\graphtemp\hbox to 0pt{\hss $u$\hss}}%
    \graphtemp=.5ex\advance\graphtemp by 0.397in
    \rlap{\kern 2.580in\lower\graphtemp\hbox to 0pt{\hss $C'$\hss}}%
    \graphtemp=.5ex\advance\graphtemp by 0.893in
    \rlap{\kern 2.580in\lower\graphtemp\hbox to 0pt{\hss $~$\hss}}%
    \graphtemp=.5ex\advance\graphtemp by 0.645in
    \rlap{\kern 3.176in\lower\graphtemp\hbox to 0pt{\hss $H$\hss}}%
    \graphtemp=.5ex\advance\graphtemp by 0.174in
    \rlap{\kern 2.431in\lower\graphtemp\hbox to 0pt{\hss $\bu$\hss}}%
    \graphtemp=.5ex\advance\graphtemp by 0.000in
    \rlap{\kern 2.183in\lower\graphtemp\hbox to 0pt{\hss $z'$\hss}}%
    \graphtemp=.5ex\advance\graphtemp by 0.074in
    \rlap{\kern 2.431in\lower\graphtemp\hbox to 0pt{\hss $z$\hss}}%
    \graphtemp=.5ex\advance\graphtemp by 0.327in
    \rlap{\kern 1.890in\lower\graphtemp\hbox to 0pt{\hss $w$\hss}}%
    \special{pa 2183 99}%
    \special{pa 2431 174}%
    \special{fp}%
    \graphtemp=.5ex\advance\graphtemp by 0.645in
    \rlap{\kern 4.069in\lower\graphtemp\hbox to 0pt{\hss $\bu$\hss}}%
    \graphtemp=.5ex\advance\graphtemp by 0.645in
    \rlap{\kern 4.317in\lower\graphtemp\hbox to 0pt{\hss $\bu$\hss}}%
    \graphtemp=.5ex\advance\graphtemp by 0.893in
    \rlap{\kern 3.821in\lower\graphtemp\hbox to 0pt{\hss $\bu$\hss}}%
    \graphtemp=.5ex\advance\graphtemp by 0.893in
    \rlap{\kern 4.565in\lower\graphtemp\hbox to 0pt{\hss $\bu$\hss}}%
    \graphtemp=.5ex\advance\graphtemp by 0.397in
    \rlap{\kern 3.821in\lower\graphtemp\hbox to 0pt{\hss $\bu$\hss}}%
    \graphtemp=.5ex\advance\graphtemp by 0.397in
    \rlap{\kern 4.565in\lower\graphtemp\hbox to 0pt{\hss $\bu$\hss}}%
    \graphtemp=.5ex\advance\graphtemp by 0.397in
    \rlap{\kern 3.573in\lower\graphtemp\hbox to 0pt{\hss $\bu$\hss}}%
    \graphtemp=.5ex\advance\graphtemp by 0.099in
    \rlap{\kern 3.796in\lower\graphtemp\hbox to 0pt{\hss $\bu$\hss}}%
    \special{pa 4069 645}%
    \special{pa 4317 645}%
    \special{fp}%
    \special{pa 4069 645}%
    \special{pa 4565 893}%
    \special{fp}%
    \special{pa 4069 645}%
    \special{pa 3821 397}%
    \special{fp}%
    \special{pa 4317 645}%
    \special{pa 4069 645}%
    \special{fp}%
    \special{pa 4317 645}%
    \special{pa 3821 893}%
    \special{fp}%
    \special{pa 4317 645}%
    \special{pa 4565 397}%
    \special{fp}%
    \special{pn 8}%
    \special{ar 4193 535 397 397 -2.786171 -0.355421}%
    \special{ar 4193 755 397 397 0.355421 2.786171}%
    \special{pn 11}%
    \special{pa 3821 397}%
    \special{pa 3573 397}%
    \special{fp}%
    \graphtemp=.5ex\advance\graphtemp by 0.670in
    \rlap{\kern 3.920in\lower\graphtemp\hbox to 0pt{\hss $x''$\hss}}%
    \graphtemp=.5ex\advance\graphtemp by 0.988in
    \rlap{\kern 3.750in\lower\graphtemp\hbox to 0pt{\hss $b$\hss}}%
    \graphtemp=.5ex\advance\graphtemp by 0.397in
    \rlap{\kern 3.920in\lower\graphtemp\hbox to 0pt{\hss $a$\hss}}%
    \graphtemp=.5ex\advance\graphtemp by 0.645in
    \rlap{\kern 4.466in\lower\graphtemp\hbox to 0pt{\hss $y''$\hss}}%
    \graphtemp=.5ex\advance\graphtemp by 0.963in
    \rlap{\kern 4.635in\lower\graphtemp\hbox to 0pt{\hss $v$\hss}}%
    \graphtemp=.5ex\advance\graphtemp by 0.327in
    \rlap{\kern 4.635in\lower\graphtemp\hbox to 0pt{\hss $u$\hss}}%
    \graphtemp=.5ex\advance\graphtemp by 0.397in
    \rlap{\kern 4.193in\lower\graphtemp\hbox to 0pt{\hss $~$\hss}}%
    \graphtemp=.5ex\advance\graphtemp by 0.893in
    \rlap{\kern 4.193in\lower\graphtemp\hbox to 0pt{\hss $~$\hss}}%
    \graphtemp=.5ex\advance\graphtemp by 0.645in
    \rlap{\kern 4.788in\lower\graphtemp\hbox to 0pt{\hss $H'$\hss}}%
    \graphtemp=.5ex\advance\graphtemp by 0.174in
    \rlap{\kern 4.044in\lower\graphtemp\hbox to 0pt{\hss $\bu$\hss}}%
    \graphtemp=.5ex\advance\graphtemp by 0.000in
    \rlap{\kern 3.796in\lower\graphtemp\hbox to 0pt{\hss $z'$\hss}}%
    \graphtemp=.5ex\advance\graphtemp by 0.074in
    \rlap{\kern 4.044in\lower\graphtemp\hbox to 0pt{\hss $z$\hss}}%
    \graphtemp=.5ex\advance\graphtemp by 0.327in
    \rlap{\kern 3.502in\lower\graphtemp\hbox to 0pt{\hss $w$\hss}}%
    \special{pa 3796 99}%
    \special{pa 4044 174}%
    \special{fp}%
    \graphtemp=.5ex\advance\graphtemp by 0.645in
    \rlap{\kern 5.681in\lower\graphtemp\hbox to 0pt{\hss $\bu$\hss}}%
    \graphtemp=.5ex\advance\graphtemp by 0.645in
    \rlap{\kern 5.929in\lower\graphtemp\hbox to 0pt{\hss $\bu$\hss}}%
    \graphtemp=.5ex\advance\graphtemp by 0.893in
    \rlap{\kern 5.433in\lower\graphtemp\hbox to 0pt{\hss $\bu$\hss}}%
    \graphtemp=.5ex\advance\graphtemp by 0.893in
    \rlap{\kern 6.177in\lower\graphtemp\hbox to 0pt{\hss $\bu$\hss}}%
    \graphtemp=.5ex\advance\graphtemp by 0.397in
    \rlap{\kern 5.433in\lower\graphtemp\hbox to 0pt{\hss $\bu$\hss}}%
    \graphtemp=.5ex\advance\graphtemp by 0.397in
    \rlap{\kern 6.177in\lower\graphtemp\hbox to 0pt{\hss $\bu$\hss}}%
    \graphtemp=.5ex\advance\graphtemp by 0.397in
    \rlap{\kern 5.185in\lower\graphtemp\hbox to 0pt{\hss $\bu$\hss}}%
    \graphtemp=.5ex\advance\graphtemp by 0.099in
    \rlap{\kern 5.408in\lower\graphtemp\hbox to 0pt{\hss $\bu$\hss}}%
    \special{pa 5681 645}%
    \special{pa 5929 645}%
    \special{fp}%
    \special{pa 5681 645}%
    \special{pa 5433 893}%
    \special{fp}%
    \special{pa 5681 645}%
    \special{pa 5433 397}%
    \special{fp}%
    \special{pa 5929 645}%
    \special{pa 5681 645}%
    \special{fp}%
    \special{pa 5929 645}%
    \special{pa 6177 893}%
    \special{fp}%
    \special{pa 5929 645}%
    \special{pa 6177 397}%
    \special{fp}%
    \special{pn 8}%
    \special{ar 5805 535 397 397 -2.786171 -0.355421}%
    \special{ar 5805 755 397 397 0.355421 2.786171}%
    \graphtemp=.5ex\advance\graphtemp by 0.645in
    \rlap{\kern 5.582in\lower\graphtemp\hbox to 0pt{\hss $x$\hss}}%
    \graphtemp=.5ex\advance\graphtemp by 0.988in
    \rlap{\kern 5.363in\lower\graphtemp\hbox to 0pt{\hss $b$\hss}}%
    \graphtemp=.5ex\advance\graphtemp by 0.397in
    \rlap{\kern 5.582in\lower\graphtemp\hbox to 0pt{\hss $a''$\hss}}%
    \graphtemp=.5ex\advance\graphtemp by 0.645in
    \rlap{\kern 6.029in\lower\graphtemp\hbox to 0pt{\hss $y$\hss}}%
    \graphtemp=.5ex\advance\graphtemp by 0.963in
    \rlap{\kern 6.248in\lower\graphtemp\hbox to 0pt{\hss $v$\hss}}%
    \graphtemp=.5ex\advance\graphtemp by 0.327in
    \rlap{\kern 6.248in\lower\graphtemp\hbox to 0pt{\hss $u$\hss}}%
    \graphtemp=.5ex\advance\graphtemp by 0.397in
    \rlap{\kern 5.855in\lower\graphtemp\hbox to 0pt{\hss $C''$\hss}}%
    \graphtemp=.5ex\advance\graphtemp by 0.893in
    \rlap{\kern 5.805in\lower\graphtemp\hbox to 0pt{\hss $D$\hss}}%
    \graphtemp=.5ex\advance\graphtemp by 0.645in
    \rlap{\kern 6.401in\lower\graphtemp\hbox to 0pt{\hss $H''$\hss}}%
    \graphtemp=.5ex\advance\graphtemp by 0.174in
    \rlap{\kern 5.656in\lower\graphtemp\hbox to 0pt{\hss $\bu$\hss}}%
    \graphtemp=.5ex\advance\graphtemp by 0.000in
    \rlap{\kern 5.408in\lower\graphtemp\hbox to 0pt{\hss $z'$\hss}}%
    \graphtemp=.5ex\advance\graphtemp by 0.074in
    \rlap{\kern 5.656in\lower\graphtemp\hbox to 0pt{\hss $z''$\hss}}%
    \graphtemp=.5ex\advance\graphtemp by 0.327in
    \rlap{\kern 5.115in\lower\graphtemp\hbox to 0pt{\hss $w$\hss}}%
    \special{pn 11}%
    \special{pa 5433 397}%
    \special{pa 5408 99}%
    \special{fp}%
    \special{pa 5185 397}%
    \special{pa 5656 174}%
    \special{fp}%
    \hbox{\vrule depth1.152in width0pt height 0pt}%
    \kern 6.500in
  }%
}%
}
\caption{Two $g$-cycles $C$ and $D$ with a common vertex.}\label{twog}
\end{figure}

\begin{lemma}\label{step4}
No two cycles of length at most $g'$ are spliced.
\end{lemma}
\begin{proof}
Already from Lemma~\ref{step0} no two $g$-cycles are spliced.

Next consider a spliced $g$-cycle $C$ and $g'$-cycle $D$ sharing the path
$\la x,y,u\ra$ such that the other neighbors of $x$ are $a$ on $C$ and $b$
on $D$.  Since $xy$ lies on a $g$-cycle, $G-\{x,y\}$ has only one alternative
reconstruction $H$, expressible so that the $g$-cycle $C'$ through the two
missing vertices $x'$ and $y'$ is obtained from $C$ by replacing $x$ with $x'$
and $y$ with $y'$.  In $H$ we then also have the edge $y'b$.  Now replacing
$\la b,x,y\ra$ with $by'$ in $D$ yields a second $g$-cycle in $H$ containing
$y'u$; this contradicts Lemma~\ref{step3}.  Hence a $g$-cycle and $g'$-cycle
cannot be spliced.

\nobreak
Now let $C$ and $D$ be two $g'$-cycles sharing $\la x,y,u\ra$, defining
$a$ and $b$ as above.  Let $v$ be the neighbor of $y$ not in $C\cup D$.
By Lemma~\ref{step0}, $C$ and $D$ cannot share three consecutive edges, so we
may let $c$ and $d$ be the neighbors of $u$ other than $y$ on $C$ and $D$,
respectively.

Consider $G-\{x,y\}$, and let $H$ be an alternative reconstruction whose
vertices deleted to form $G-\{x,y\}$ are $x'$ and $y'$.  We know $x'y'\in E(H)$,
and we can label $x'$ and $y'$ so that $y'u,x'v\in E(H)$.  The remaining
neighbor of $x'$ may be $a$ or $b$, but since $C$ and $D$ are both $g'$-cycles
these choices are symmetric.  Hence we may assume $ax',by'\in E(H)$
(see Figure~\ref{splice}).

Replacing $\la b,x,y,u\ra$ in $D$ with $\la b, y',u\ra$ yields a $g$-cycle
$D'$ in $H$ containing exactly one of $\{x',y'\}$.  Hence $G$ must have a
$g$-cycle $Q$ containing exactly one of $\{x,y\}$.  Since $Q$ cannot contain
$xy$, it contains $\la v,y,u\ra$ or $\la a,x,b\ra$.  In the first case,
continuing $Q$ along the edge leaving $u$ in either $C$ or $D$ yields two
spliced $g$-cycles ($Q$ with $C$ or $D$), which is forbidden.  Hence $Q$
contains $\la a,x,b\ra$.

Applying the symmetric argument to $G-\{u,y\}$ allows us to conclude that
$G$ also contains a $g$-cycle through $\la c,u,d\ra$.  This $g$-cycle $R$ also
appears in $H$.  Now $R$ and $D'$ are $g$-cycles in $H$ that both contain the
edge $ud$.  By Lemma~\ref{step3}, this is forbidden.
\end{proof}

\begin{lemma}\label{step5}
There is no edge whose endpoints lie in distinct $g$-cycles.
\end{lemma}
\begin{proof}
Let $xy$ be such an edge in $G$, joining $g$-cycles $C$ through $x$ and $D$
through $y$.  In any alternative reconstruction $H$ from $G-\{x,y\}$, the
missing vertices $x'$ and $y'$ are adjacent, and to avoid recreating $G$ each
of $x'$ and $y'$ must have one neighbor in $V(C)$ and one neighbor in $V(D)$.
Each possible assignment yields in $H$ two $g'$-cycles containing $x'y'$.  By
symmetry, we may label $C-x$ as an $a,b$-path and $D-y$ as a $u,v$-path
so that $a,u\in N_H(x')$ and $b,v\in N_H(y')$.  See Figure~\ref{distance1},
where we have not yet established the dashed edges.

Since $G$ had two $g$-cycles each containing exactly one of $\{x,y\}$, also $H$
must have two $g$-cycles each containing exactly one of $\{x',y'\}$.  One must
use $\la a,x',u\ra$, and the other must use $\la b,y',v\ra$; let these be $Q'$
and $R'$, respectively.  Replacing $\la a,x',u\ra$ in $Q'$ with $\la a,x,y,u\ra$
and $\la b,y',v\ra$ in $R'$ with $\la b,x,y,v\ra$ yields $g'$-cycles $Q$ and
$R$ in $G$, respectively.

Now let $z$ and $w$ be the neighbors of $a$ other than $x$ in $Q$ and $C$,
respectively, and let $t$ be the neighbor of $z$ on $Q$ other than $a$.
Consider an alternative reconstruction $H'$ from $G-\{a,z\}$, with $a'$ and
$z'$ being the missing vertices.  We have $a'z'\in E(H')$.  By symmetry we may
assume $a'x\in E(H')$, and hence $wz'\in E(H')$ to avoid recreating $G$.  Still
$t$ may be ajacent to $z'$ or to $a'$.  The edge $z't$ would complete a
$g'$-cycle $Q''$ in $H'$ (shown below) that is spliced with $C'$, forbidden by
Lemma~\ref{step4}.  The edge $a't$ would complete a $g$-cycle in $H'$ sharing
the edge $uy$ with the $g$-cycle $R$, forbidden by Lemma~\ref{step3}.
\end{proof}

\vspace{-2pc}
\begin{figure}[hbt]
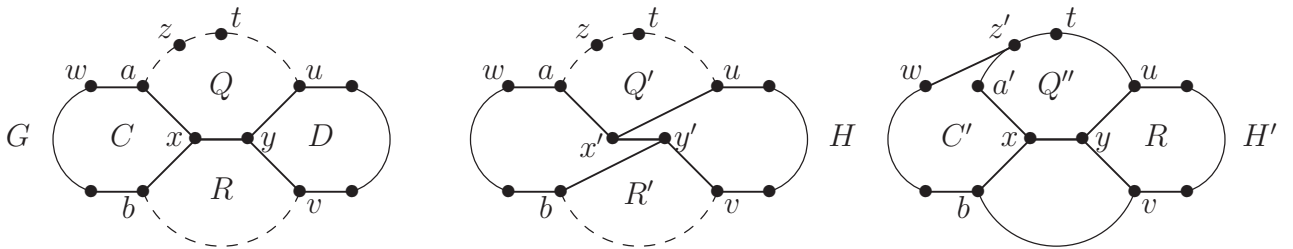

\gpic{
\expandafter\ifx\csname graph\endcsname\relax \csname newbox\endcsname\graph\fi
\expandafter\ifx\csname graphtemp\endcsname\relax \csname newdimen\endcsname\graphtemp\fi
\setbox\graph=\vtop{\vskip 0pt\hbox{%
    \graphtemp=.5ex\advance\graphtemp by 0.655in
    \rlap{\kern 0.929in\lower\graphtemp\hbox to 0pt{\hss $\bu$\hss}}%
    \graphtemp=.5ex\advance\graphtemp by 0.655in
    \rlap{\kern 1.202in\lower\graphtemp\hbox to 0pt{\hss $\bu$\hss}}%
    \graphtemp=.5ex\advance\graphtemp by 0.929in
    \rlap{\kern 0.655in\lower\graphtemp\hbox to 0pt{\hss $\bu$\hss}}%
    \graphtemp=.5ex\advance\graphtemp by 0.929in
    \rlap{\kern 1.475in\lower\graphtemp\hbox to 0pt{\hss $\bu$\hss}}%
    \graphtemp=.5ex\advance\graphtemp by 0.382in
    \rlap{\kern 0.655in\lower\graphtemp\hbox to 0pt{\hss $\bu$\hss}}%
    \graphtemp=.5ex\advance\graphtemp by 0.382in
    \rlap{\kern 1.475in\lower\graphtemp\hbox to 0pt{\hss $\bu$\hss}}%
    \graphtemp=.5ex\advance\graphtemp by 0.929in
    \rlap{\kern 0.382in\lower\graphtemp\hbox to 0pt{\hss $\bu$\hss}}%
    \graphtemp=.5ex\advance\graphtemp by 0.929in
    \rlap{\kern 1.748in\lower\graphtemp\hbox to 0pt{\hss $\bu$\hss}}%
    \graphtemp=.5ex\advance\graphtemp by 0.382in
    \rlap{\kern 0.382in\lower\graphtemp\hbox to 0pt{\hss $\bu$\hss}}%
    \graphtemp=.5ex\advance\graphtemp by 0.382in
    \rlap{\kern 1.748in\lower\graphtemp\hbox to 0pt{\hss $\bu$\hss}}%
    \special{pn 11}%
    \special{pa 929 655}%
    \special{pa 1202 655}%
    \special{fp}%
    \special{pa 929 655}%
    \special{pa 655 929}%
    \special{fp}%
    \special{pa 929 655}%
    \special{pa 655 382}%
    \special{fp}%
    \special{pa 1202 655}%
    \special{pa 929 655}%
    \special{fp}%
    \special{pa 1202 655}%
    \special{pa 1475 929}%
    \special{fp}%
    \special{pa 1202 655}%
    \special{pa 1475 382}%
    \special{fp}%
    \special{pn 8}%
    \special{ar 1065 534 437 437 -0.480421 -0.355421}%
    \special{ar 1065 534 437 437 -0.736616 -0.611616}%
    \special{ar 1065 534 437 437 -0.992810 -0.867810}%
    \special{ar 1065 534 437 437 -1.249005 -1.124005}%
    \special{ar 1065 534 437 437 -1.505199 -1.380199}%
    \special{ar 1065 534 437 437 -1.761394 -1.636394}%
    \special{ar 1065 534 437 437 -2.017588 -1.892588}%
    \special{ar 1065 534 437 437 -2.273783 -2.148783}%
    \special{ar 1065 534 437 437 -2.529977 -2.404977}%
    \special{ar 1065 534 437 437 -2.786171 -2.661171}%
    \special{ar 1065 777 437 437 2.661171 2.786171}%
    \special{ar 1065 777 437 437 2.404977 2.529977}%
    \special{ar 1065 777 437 437 2.148783 2.273783}%
    \special{ar 1065 777 437 437 1.892588 2.017588}%
    \special{ar 1065 777 437 437 1.636394 1.761394}%
    \special{ar 1065 777 437 437 1.380199 1.505199}%
    \special{ar 1065 777 437 437 1.124005 1.249005}%
    \special{ar 1065 777 437 437 0.867810 0.992810}%
    \special{ar 1065 777 437 437 0.611616 0.736616}%
    \special{ar 1065 777 437 437 0.355421 0.480421}%
    \special{ar 470 655 287 287 1.880641 4.402544}%
    \special{ar 1660 655 287 287 -1.260952 1.260952}%
    \special{pn 11}%
    \special{pa 655 929}%
    \special{pa 382 929}%
    \special{fp}%
    \special{pa 1475 929}%
    \special{pa 1748 929}%
    \special{fp}%
    \special{pa 655 382}%
    \special{pa 382 382}%
    \special{fp}%
    \special{pa 1475 382}%
    \special{pa 1748 382}%
    \special{fp}%
    \graphtemp=.5ex\advance\graphtemp by 0.655in
    \rlap{\kern 0.819in\lower\graphtemp\hbox to 0pt{\hss $x$\hss}}%
    \graphtemp=.5ex\advance\graphtemp by 1.033in
    \rlap{\kern 0.578in\lower\graphtemp\hbox to 0pt{\hss $b$\hss}}%
    \graphtemp=.5ex\advance\graphtemp by 0.305in
    \rlap{\kern 0.578in\lower\graphtemp\hbox to 0pt{\hss $a$\hss}}%
    \graphtemp=.5ex\advance\graphtemp by 0.655in
    \rlap{\kern 1.311in\lower\graphtemp\hbox to 0pt{\hss $y$\hss}}%
    \graphtemp=.5ex\advance\graphtemp by 1.006in
    \rlap{\kern 1.552in\lower\graphtemp\hbox to 0pt{\hss $v$\hss}}%
    \graphtemp=.5ex\advance\graphtemp by 0.305in
    \rlap{\kern 1.552in\lower\graphtemp\hbox to 0pt{\hss $u$\hss}}%
    \graphtemp=.5ex\advance\graphtemp by 0.382in
    \rlap{\kern 1.065in\lower\graphtemp\hbox to 0pt{\hss $Q$\hss}}%
    \graphtemp=.5ex\advance\graphtemp by 0.929in
    \rlap{\kern 1.065in\lower\graphtemp\hbox to 0pt{\hss $R$\hss}}%
    \graphtemp=.5ex\advance\graphtemp by 0.655in
    \rlap{\kern 0.000in\lower\graphtemp\hbox to 0pt{\hss $G$\hss}}%
    \graphtemp=.5ex\advance\graphtemp by 1.033in
    \rlap{\kern 0.305in\lower\graphtemp\hbox to 0pt{\hss $~$\hss}}%
    \graphtemp=.5ex\advance\graphtemp by 0.305in
    \rlap{\kern 0.305in\lower\graphtemp\hbox to 0pt{\hss $w$\hss}}%
    \graphtemp=.5ex\advance\graphtemp by 1.006in
    \rlap{\kern 1.825in\lower\graphtemp\hbox to 0pt{\hss $~$\hss}}%
    \graphtemp=.5ex\advance\graphtemp by 0.655in
    \rlap{\kern 0.546in\lower\graphtemp\hbox to 0pt{\hss $C$\hss}}%
    \graphtemp=.5ex\advance\graphtemp by 0.655in
    \rlap{\kern 1.584in\lower\graphtemp\hbox to 0pt{\hss $D$\hss}}%
    \graphtemp=.5ex\advance\graphtemp by 0.305in
    \rlap{\kern 1.825in\lower\graphtemp\hbox to 0pt{\hss $~$\hss}}%
    \graphtemp=.5ex\advance\graphtemp by 0.164in
    \rlap{\kern 0.847in\lower\graphtemp\hbox to 0pt{\hss $\bu$\hss}}%
    \graphtemp=.5ex\advance\graphtemp by 0.087in
    \rlap{\kern 0.769in\lower\graphtemp\hbox to 0pt{\hss $z$\hss}}%
    \graphtemp=.5ex\advance\graphtemp by 0.109in
    \rlap{\kern 1.065in\lower\graphtemp\hbox to 0pt{\hss $\bu$\hss}}%
    \graphtemp=.5ex\advance\graphtemp by 0.032in
    \rlap{\kern 1.142in\lower\graphtemp\hbox to 0pt{\hss $t$\hss}}%
    \graphtemp=.5ex\advance\graphtemp by 0.655in
    \rlap{\kern 3.113in\lower\graphtemp\hbox to 0pt{\hss $\bu$\hss}}%
    \graphtemp=.5ex\advance\graphtemp by 0.655in
    \rlap{\kern 3.387in\lower\graphtemp\hbox to 0pt{\hss $\bu$\hss}}%
    \graphtemp=.5ex\advance\graphtemp by 0.929in
    \rlap{\kern 2.840in\lower\graphtemp\hbox to 0pt{\hss $\bu$\hss}}%
    \graphtemp=.5ex\advance\graphtemp by 0.929in
    \rlap{\kern 3.660in\lower\graphtemp\hbox to 0pt{\hss $\bu$\hss}}%
    \graphtemp=.5ex\advance\graphtemp by 0.382in
    \rlap{\kern 2.840in\lower\graphtemp\hbox to 0pt{\hss $\bu$\hss}}%
    \graphtemp=.5ex\advance\graphtemp by 0.382in
    \rlap{\kern 3.660in\lower\graphtemp\hbox to 0pt{\hss $\bu$\hss}}%
    \graphtemp=.5ex\advance\graphtemp by 0.929in
    \rlap{\kern 2.567in\lower\graphtemp\hbox to 0pt{\hss $\bu$\hss}}%
    \graphtemp=.5ex\advance\graphtemp by 0.929in
    \rlap{\kern 3.933in\lower\graphtemp\hbox to 0pt{\hss $\bu$\hss}}%
    \graphtemp=.5ex\advance\graphtemp by 0.382in
    \rlap{\kern 2.567in\lower\graphtemp\hbox to 0pt{\hss $\bu$\hss}}%
    \graphtemp=.5ex\advance\graphtemp by 0.382in
    \rlap{\kern 3.933in\lower\graphtemp\hbox to 0pt{\hss $\bu$\hss}}%
    \special{pa 3113 655}%
    \special{pa 3387 655}%
    \special{fp}%
    \special{pa 3113 655}%
    \special{pa 3660 382}%
    \special{fp}%
    \special{pa 3113 655}%
    \special{pa 2840 382}%
    \special{fp}%
    \special{pa 3387 655}%
    \special{pa 3113 655}%
    \special{fp}%
    \special{pa 3387 655}%
    \special{pa 2840 929}%
    \special{fp}%
    \special{pa 3387 655}%
    \special{pa 3660 929}%
    \special{fp}%
    \special{pn 8}%
    \special{ar 3250 534 437 437 -0.480421 -0.355421}%
    \special{ar 3250 534 437 437 -0.736616 -0.611616}%
    \special{ar 3250 534 437 437 -0.992810 -0.867810}%
    \special{ar 3250 534 437 437 -1.249005 -1.124005}%
    \special{ar 3250 534 437 437 -1.505199 -1.380199}%
    \special{ar 3250 534 437 437 -1.761394 -1.636394}%
    \special{ar 3250 534 437 437 -2.017588 -1.892588}%
    \special{ar 3250 534 437 437 -2.273783 -2.148783}%
    \special{ar 3250 534 437 437 -2.529977 -2.404977}%
    \special{ar 3250 534 437 437 -2.786171 -2.661171}%
    \special{ar 3250 777 437 437 2.661171 2.786171}%
    \special{ar 3250 777 437 437 2.404977 2.529977}%
    \special{ar 3250 777 437 437 2.148783 2.273783}%
    \special{ar 3250 777 437 437 1.892588 2.017588}%
    \special{ar 3250 777 437 437 1.636394 1.761394}%
    \special{ar 3250 777 437 437 1.380199 1.505199}%
    \special{ar 3250 777 437 437 1.124005 1.249005}%
    \special{ar 3250 777 437 437 0.867810 0.992810}%
    \special{ar 3250 777 437 437 0.611616 0.736616}%
    \special{ar 3250 777 437 437 0.355421 0.480421}%
    \special{ar 2655 655 287 287 1.880641 4.402544}%
    \special{ar 3845 655 287 287 -1.260952 1.260952}%
    \special{pn 11}%
    \special{pa 2840 929}%
    \special{pa 2567 929}%
    \special{fp}%
    \special{pa 3660 929}%
    \special{pa 3933 929}%
    \special{fp}%
    \special{pa 2840 382}%
    \special{pa 2567 382}%
    \special{fp}%
    \special{pa 3660 382}%
    \special{pa 3933 382}%
    \special{fp}%
    \graphtemp=.5ex\advance\graphtemp by 0.710in
    \rlap{\kern 3.004in\lower\graphtemp\hbox to 0pt{\hss $x'$\hss}}%
    \graphtemp=.5ex\advance\graphtemp by 1.033in
    \rlap{\kern 2.763in\lower\graphtemp\hbox to 0pt{\hss $b$\hss}}%
    \graphtemp=.5ex\advance\graphtemp by 0.305in
    \rlap{\kern 2.763in\lower\graphtemp\hbox to 0pt{\hss $a$\hss}}%
    \graphtemp=.5ex\advance\graphtemp by 0.655in
    \rlap{\kern 3.496in\lower\graphtemp\hbox to 0pt{\hss $y'$\hss}}%
    \graphtemp=.5ex\advance\graphtemp by 1.006in
    \rlap{\kern 3.737in\lower\graphtemp\hbox to 0pt{\hss $v$\hss}}%
    \graphtemp=.5ex\advance\graphtemp by 0.305in
    \rlap{\kern 3.737in\lower\graphtemp\hbox to 0pt{\hss $u$\hss}}%
    \graphtemp=.5ex\advance\graphtemp by 0.382in
    \rlap{\kern 3.250in\lower\graphtemp\hbox to 0pt{\hss $Q'$\hss}}%
    \graphtemp=.5ex\advance\graphtemp by 0.956in
    \rlap{\kern 3.250in\lower\graphtemp\hbox to 0pt{\hss $R'$\hss}}%
    \graphtemp=.5ex\advance\graphtemp by 0.655in
    \rlap{\kern 4.315in\lower\graphtemp\hbox to 0pt{\hss $H$\hss}}%
    \graphtemp=.5ex\advance\graphtemp by 1.033in
    \rlap{\kern 2.490in\lower\graphtemp\hbox to 0pt{\hss $~$\hss}}%
    \graphtemp=.5ex\advance\graphtemp by 0.305in
    \rlap{\kern 2.490in\lower\graphtemp\hbox to 0pt{\hss $w$\hss}}%
    \graphtemp=.5ex\advance\graphtemp by 1.006in
    \rlap{\kern 4.010in\lower\graphtemp\hbox to 0pt{\hss $~$\hss}}%
    \graphtemp=.5ex\advance\graphtemp by 0.655in
    \rlap{\kern 2.731in\lower\graphtemp\hbox to 0pt{\hss $~$\hss}}%
    \graphtemp=.5ex\advance\graphtemp by 0.655in
    \rlap{\kern 3.769in\lower\graphtemp\hbox to 0pt{\hss $~$\hss}}%
    \graphtemp=.5ex\advance\graphtemp by 0.305in
    \rlap{\kern 4.010in\lower\graphtemp\hbox to 0pt{\hss $~$\hss}}%
    \graphtemp=.5ex\advance\graphtemp by 0.164in
    \rlap{\kern 3.032in\lower\graphtemp\hbox to 0pt{\hss $\bu$\hss}}%
    \graphtemp=.5ex\advance\graphtemp by 0.087in
    \rlap{\kern 2.954in\lower\graphtemp\hbox to 0pt{\hss $z$\hss}}%
    \graphtemp=.5ex\advance\graphtemp by 0.109in
    \rlap{\kern 3.250in\lower\graphtemp\hbox to 0pt{\hss $\bu$\hss}}%
    \graphtemp=.5ex\advance\graphtemp by 0.032in
    \rlap{\kern 3.327in\lower\graphtemp\hbox to 0pt{\hss $t$\hss}}%
    \graphtemp=.5ex\advance\graphtemp by 0.655in
    \rlap{\kern 5.298in\lower\graphtemp\hbox to 0pt{\hss $\bu$\hss}}%
    \graphtemp=.5ex\advance\graphtemp by 0.655in
    \rlap{\kern 5.571in\lower\graphtemp\hbox to 0pt{\hss $\bu$\hss}}%
    \graphtemp=.5ex\advance\graphtemp by 0.929in
    \rlap{\kern 5.025in\lower\graphtemp\hbox to 0pt{\hss $\bu$\hss}}%
    \graphtemp=.5ex\advance\graphtemp by 0.929in
    \rlap{\kern 5.845in\lower\graphtemp\hbox to 0pt{\hss $\bu$\hss}}%
    \graphtemp=.5ex\advance\graphtemp by 0.382in
    \rlap{\kern 5.025in\lower\graphtemp\hbox to 0pt{\hss $\bu$\hss}}%
    \graphtemp=.5ex\advance\graphtemp by 0.382in
    \rlap{\kern 5.845in\lower\graphtemp\hbox to 0pt{\hss $\bu$\hss}}%
    \graphtemp=.5ex\advance\graphtemp by 0.929in
    \rlap{\kern 4.752in\lower\graphtemp\hbox to 0pt{\hss $\bu$\hss}}%
    \graphtemp=.5ex\advance\graphtemp by 0.929in
    \rlap{\kern 6.118in\lower\graphtemp\hbox to 0pt{\hss $\bu$\hss}}%
    \graphtemp=.5ex\advance\graphtemp by 0.382in
    \rlap{\kern 4.752in\lower\graphtemp\hbox to 0pt{\hss $\bu$\hss}}%
    \graphtemp=.5ex\advance\graphtemp by 0.382in
    \rlap{\kern 6.118in\lower\graphtemp\hbox to 0pt{\hss $\bu$\hss}}%
    \special{pa 5298 655}%
    \special{pa 5571 655}%
    \special{fp}%
    \special{pa 5298 655}%
    \special{pa 5025 929}%
    \special{fp}%
    \special{pa 5298 655}%
    \special{pa 5025 382}%
    \special{fp}%
    \special{pa 5571 655}%
    \special{pa 5298 655}%
    \special{fp}%
    \special{pa 5571 655}%
    \special{pa 5845 929}%
    \special{fp}%
    \special{pa 5571 655}%
    \special{pa 5845 382}%
    \special{fp}%
    \special{pn 8}%
    \special{ar 5435 534 437 437 -2.786171 -0.355421}%
    \special{ar 5435 777 437 437 0.355421 2.786171}%
    \special{ar 4840 655 287 287 1.880641 4.402544}%
    \special{ar 6030 655 287 287 -1.260952 1.260952}%
    \graphtemp=.5ex\advance\graphtemp by 0.655in
    \rlap{\kern 5.189in\lower\graphtemp\hbox to 0pt{\hss $x$\hss}}%
    \graphtemp=.5ex\advance\graphtemp by 1.033in
    \rlap{\kern 4.948in\lower\graphtemp\hbox to 0pt{\hss $b$\hss}}%
    \graphtemp=.5ex\advance\graphtemp by 0.382in
    \rlap{\kern 5.162in\lower\graphtemp\hbox to 0pt{\hss $a'$\hss}}%
    \graphtemp=.5ex\advance\graphtemp by 0.655in
    \rlap{\kern 5.681in\lower\graphtemp\hbox to 0pt{\hss $y$\hss}}%
    \graphtemp=.5ex\advance\graphtemp by 1.006in
    \rlap{\kern 5.922in\lower\graphtemp\hbox to 0pt{\hss $v$\hss}}%
    \graphtemp=.5ex\advance\graphtemp by 0.305in
    \rlap{\kern 5.922in\lower\graphtemp\hbox to 0pt{\hss $u$\hss}}%
    \graphtemp=.5ex\advance\graphtemp by 0.382in
    \rlap{\kern 5.435in\lower\graphtemp\hbox to 0pt{\hss $Q''$\hss}}%
    \graphtemp=.5ex\advance\graphtemp by 0.929in
    \rlap{\kern 5.435in\lower\graphtemp\hbox to 0pt{\hss $~$\hss}}%
    \graphtemp=.5ex\advance\graphtemp by 0.655in
    \rlap{\kern 6.500in\lower\graphtemp\hbox to 0pt{\hss $H'$\hss}}%
    \graphtemp=.5ex\advance\graphtemp by 1.033in
    \rlap{\kern 4.675in\lower\graphtemp\hbox to 0pt{\hss $~$\hss}}%
    \graphtemp=.5ex\advance\graphtemp by 0.305in
    \rlap{\kern 4.675in\lower\graphtemp\hbox to 0pt{\hss $w$\hss}}%
    \graphtemp=.5ex\advance\graphtemp by 1.006in
    \rlap{\kern 6.195in\lower\graphtemp\hbox to 0pt{\hss $~$\hss}}%
    \graphtemp=.5ex\advance\graphtemp by 0.655in
    \rlap{\kern 4.916in\lower\graphtemp\hbox to 0pt{\hss $C'$\hss}}%
    \graphtemp=.5ex\advance\graphtemp by 0.655in
    \rlap{\kern 5.954in\lower\graphtemp\hbox to 0pt{\hss $R$\hss}}%
    \graphtemp=.5ex\advance\graphtemp by 0.305in
    \rlap{\kern 6.195in\lower\graphtemp\hbox to 0pt{\hss $~$\hss}}%
    \graphtemp=.5ex\advance\graphtemp by 0.164in
    \rlap{\kern 5.216in\lower\graphtemp\hbox to 0pt{\hss $\bu$\hss}}%
    \graphtemp=.5ex\advance\graphtemp by 0.087in
    \rlap{\kern 5.139in\lower\graphtemp\hbox to 0pt{\hss $z'$\hss}}%
    \special{pn 11}%
    \special{pa 5025 929}%
    \special{pa 4752 929}%
    \special{fp}%
    \special{pa 5845 929}%
    \special{pa 6118 929}%
    \special{fp}%
    \special{pa 5845 382}%
    \special{pa 6118 382}%
    \special{fp}%
    \special{pa 4752 382}%
    \special{pa 5216 164}%
    \special{fp}%
    \graphtemp=.5ex\advance\graphtemp by 0.109in
    \rlap{\kern 5.435in\lower\graphtemp\hbox to 0pt{\hss $\bu$\hss}}%
    \graphtemp=.5ex\advance\graphtemp by 0.032in
    \rlap{\kern 5.512in\lower\graphtemp\hbox to 0pt{\hss $t$\hss}}%
    \hbox{\vrule depth1.213in width0pt height 0pt}%
    \kern 6.500in
  }%
}%
}
\caption{Adjacent vertices in $g$-cycles.}\label{distance1}
\end{figure}

\begin{lemma}\label{step6}
No vertex lies in a $g$-cycle and two $g'$-cycles.
\end{lemma}
\begin{proof}
Suppose that $x$ is such a vertex in the $3$-regular graph $G$.  Each of the
three cycles uses two edges incident to $x$.  Since there are no spliced cycles
of length at most $g'$, each remaining edge incident to the neighbors of $x$
lies in exactly one of these cycles.  Let $N_G(x)=\{y,a,b\}$ and
$N_G(y)=\{x,u,v\}$, with $xy$ shared by the $g'$-cycles $Q$ and $R$, and with
$a,u\in V(Q)$ and $b,v\in V(R)$ (see Figure~\ref{threecyc}).
The $g$-cycle $C$ through $x$ contains $\la a,x,b\ra$.

Let $H$ be an alternative reconstruction from $G-\{x,y\}$, with missing
vertices $x'$ and $y'$.  We have $x'y'\in E(H)$, and by 
symmetry we may assume $ax'\in E(H)$ and hence also $by'\in E(H)$.
Now there are two cases, shown in Figure~\ref{threecyc}.  
If $ux',vy'\in E(H)$, then replacing $\la a,x,y,u\ra$ in $Q$ with
$\la a,x',u\ra$ and replacing $\la b,x,y,v\ra$ in $R$ with $\la b,y',v\ra$
yields $g$-cycles $Q'$ and $R'$ in $H$ through the endpoints of the edge
$x'y'$, which is forbidden by Lemma~\ref{step5}.  If $uy',vx'\in E(H)$, then
replacing $\la a,x,b\ra$ in $C$ with $\la a,x',y',b\ra$ yields a $g'$-cycle
$C'$ in $H$ that is spliced with the $g'$-cycles $Q''$ through
$\la a,x',y',u\ra$ and $R''$ through $\la b,y',x',v\ra$, which is forbidden by
Lemma~\ref{step4}.
\end{proof}

\vspace{-1pc}
\begin{figure}[hbt]
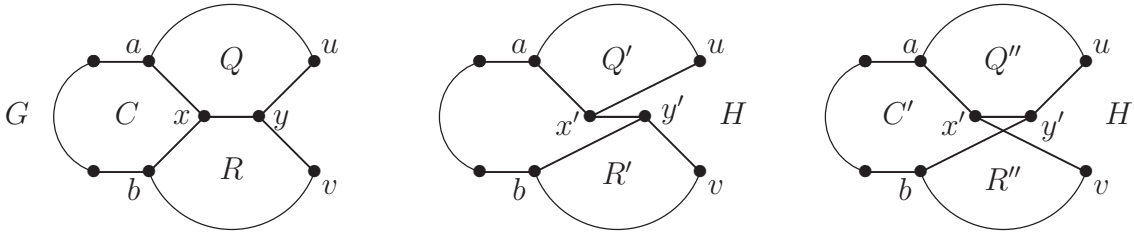

\gpic{
\expandafter\ifx\csname graph\endcsname\relax \csname newbox\endcsname\graph\fi
\expandafter\ifx\csname graphtemp\endcsname\relax \csname newdimen\endcsname\graphtemp\fi
\setbox\graph=\vtop{\vskip 0pt\hbox{%
    \graphtemp=.5ex\advance\graphtemp by 0.589in
    \rlap{\kern 0.981in\lower\graphtemp\hbox to 0pt{\hss $\bu$\hss}}%
    \graphtemp=.5ex\advance\graphtemp by 0.589in
    \rlap{\kern 1.269in\lower\graphtemp\hbox to 0pt{\hss $\bu$\hss}}%
    \graphtemp=.5ex\advance\graphtemp by 0.878in
    \rlap{\kern 0.692in\lower\graphtemp\hbox to 0pt{\hss $\bu$\hss}}%
    \graphtemp=.5ex\advance\graphtemp by 0.878in
    \rlap{\kern 1.558in\lower\graphtemp\hbox to 0pt{\hss $\bu$\hss}}%
    \graphtemp=.5ex\advance\graphtemp by 0.301in
    \rlap{\kern 0.692in\lower\graphtemp\hbox to 0pt{\hss $\bu$\hss}}%
    \graphtemp=.5ex\advance\graphtemp by 0.301in
    \rlap{\kern 1.558in\lower\graphtemp\hbox to 0pt{\hss $\bu$\hss}}%
    \graphtemp=.5ex\advance\graphtemp by 0.878in
    \rlap{\kern 0.404in\lower\graphtemp\hbox to 0pt{\hss $\bu$\hss}}%
    \graphtemp=.5ex\advance\graphtemp by 0.301in
    \rlap{\kern 0.404in\lower\graphtemp\hbox to 0pt{\hss $\bu$\hss}}%
    \graphtemp=.5ex\advance\graphtemp by 0.301in
    \rlap{\kern 0.404in\lower\graphtemp\hbox to 0pt{\hss $\bu$\hss}}%
    \special{pn 11}%
    \special{pa 981 589}%
    \special{pa 1269 589}%
    \special{fp}%
    \special{pa 981 589}%
    \special{pa 692 878}%
    \special{fp}%
    \special{pa 981 589}%
    \special{pa 692 301}%
    \special{fp}%
    \special{pa 1269 589}%
    \special{pa 981 589}%
    \special{fp}%
    \special{pa 1269 589}%
    \special{pa 1558 878}%
    \special{fp}%
    \special{pa 1269 589}%
    \special{pa 1558 301}%
    \special{fp}%
    \special{pn 8}%
    \special{ar 1125 462 462 462 -2.786171 -0.355421}%
    \special{ar 1125 717 462 462 0.355421 2.786171}%
    \special{ar 496 589 303 303 1.880641 4.402544}%
    \special{pn 11}%
    \special{pa 692 878}%
    \special{pa 404 878}%
    \special{fp}%
    \special{pa 692 301}%
    \special{pa 404 301}%
    \special{fp}%
    \graphtemp=.5ex\advance\graphtemp by 0.589in
    \rlap{\kern 0.865in\lower\graphtemp\hbox to 0pt{\hss $x$\hss}}%
    \graphtemp=.5ex\advance\graphtemp by 0.988in
    \rlap{\kern 0.611in\lower\graphtemp\hbox to 0pt{\hss $b$\hss}}%
    \graphtemp=.5ex\advance\graphtemp by 0.219in
    \rlap{\kern 0.611in\lower\graphtemp\hbox to 0pt{\hss $a$\hss}}%
    \graphtemp=.5ex\advance\graphtemp by 0.589in
    \rlap{\kern 1.385in\lower\graphtemp\hbox to 0pt{\hss $y$\hss}}%
    \graphtemp=.5ex\advance\graphtemp by 0.959in
    \rlap{\kern 1.639in\lower\graphtemp\hbox to 0pt{\hss $v$\hss}}%
    \graphtemp=.5ex\advance\graphtemp by 0.219in
    \rlap{\kern 1.639in\lower\graphtemp\hbox to 0pt{\hss $u$\hss}}%
    \graphtemp=.5ex\advance\graphtemp by 0.301in
    \rlap{\kern 1.125in\lower\graphtemp\hbox to 0pt{\hss $Q$\hss}}%
    \graphtemp=.5ex\advance\graphtemp by 0.878in
    \rlap{\kern 1.125in\lower\graphtemp\hbox to 0pt{\hss $R$\hss}}%
    \graphtemp=.5ex\advance\graphtemp by 0.589in
    \rlap{\kern 0.000in\lower\graphtemp\hbox to 0pt{\hss $G$\hss}}%
    \graphtemp=.5ex\advance\graphtemp by 0.589in
    \rlap{\kern 0.577in\lower\graphtemp\hbox to 0pt{\hss $C$\hss}}%
    \graphtemp=.5ex\advance\graphtemp by 0.589in
    \rlap{\kern 3.000in\lower\graphtemp\hbox to 0pt{\hss $\bu$\hss}}%
    \graphtemp=.5ex\advance\graphtemp by 0.589in
    \rlap{\kern 3.288in\lower\graphtemp\hbox to 0pt{\hss $\bu$\hss}}%
    \graphtemp=.5ex\advance\graphtemp by 0.878in
    \rlap{\kern 2.712in\lower\graphtemp\hbox to 0pt{\hss $\bu$\hss}}%
    \graphtemp=.5ex\advance\graphtemp by 0.878in
    \rlap{\kern 3.577in\lower\graphtemp\hbox to 0pt{\hss $\bu$\hss}}%
    \graphtemp=.5ex\advance\graphtemp by 0.301in
    \rlap{\kern 2.712in\lower\graphtemp\hbox to 0pt{\hss $\bu$\hss}}%
    \graphtemp=.5ex\advance\graphtemp by 0.301in
    \rlap{\kern 3.577in\lower\graphtemp\hbox to 0pt{\hss $\bu$\hss}}%
    \graphtemp=.5ex\advance\graphtemp by 0.878in
    \rlap{\kern 2.423in\lower\graphtemp\hbox to 0pt{\hss $\bu$\hss}}%
    \graphtemp=.5ex\advance\graphtemp by 0.301in
    \rlap{\kern 2.423in\lower\graphtemp\hbox to 0pt{\hss $\bu$\hss}}%
    \graphtemp=.5ex\advance\graphtemp by 0.301in
    \rlap{\kern 2.423in\lower\graphtemp\hbox to 0pt{\hss $\bu$\hss}}%
    \special{pa 3000 589}%
    \special{pa 3288 589}%
    \special{fp}%
    \special{pa 3000 589}%
    \special{pa 3577 301}%
    \special{fp}%
    \special{pa 3000 589}%
    \special{pa 2712 301}%
    \special{fp}%
    \special{pa 3288 589}%
    \special{pa 3000 589}%
    \special{fp}%
    \special{pa 3288 589}%
    \special{pa 2712 878}%
    \special{fp}%
    \special{pa 3288 589}%
    \special{pa 3577 878}%
    \special{fp}%
    \special{pn 8}%
    \special{ar 3144 462 462 462 -2.786171 -0.355421}%
    \special{ar 3144 717 462 462 0.355421 2.786171}%
    \special{ar 2515 589 303 303 1.880641 4.402544}%
    \special{pn 11}%
    \special{pa 2712 878}%
    \special{pa 2423 878}%
    \special{fp}%
    \special{pa 2712 301}%
    \special{pa 2423 301}%
    \special{fp}%
    \graphtemp=.5ex\advance\graphtemp by 0.647in
    \rlap{\kern 2.885in\lower\graphtemp\hbox to 0pt{\hss $x'$\hss}}%
    \graphtemp=.5ex\advance\graphtemp by 0.988in
    \rlap{\kern 2.630in\lower\graphtemp\hbox to 0pt{\hss $b$\hss}}%
    \graphtemp=.5ex\advance\graphtemp by 0.219in
    \rlap{\kern 2.630in\lower\graphtemp\hbox to 0pt{\hss $a$\hss}}%
    \graphtemp=.5ex\advance\graphtemp by 0.561in
    \rlap{\kern 3.433in\lower\graphtemp\hbox to 0pt{\hss $y'$\hss}}%
    \graphtemp=.5ex\advance\graphtemp by 0.959in
    \rlap{\kern 3.659in\lower\graphtemp\hbox to 0pt{\hss $v$\hss}}%
    \graphtemp=.5ex\advance\graphtemp by 0.219in
    \rlap{\kern 3.659in\lower\graphtemp\hbox to 0pt{\hss $u$\hss}}%
    \graphtemp=.5ex\advance\graphtemp by 0.301in
    \rlap{\kern 3.144in\lower\graphtemp\hbox to 0pt{\hss $Q'$\hss}}%
    \graphtemp=.5ex\advance\graphtemp by 0.907in
    \rlap{\kern 3.144in\lower\graphtemp\hbox to 0pt{\hss $R'$\hss}}%
    \graphtemp=.5ex\advance\graphtemp by 0.589in
    \rlap{\kern 3.750in\lower\graphtemp\hbox to 0pt{\hss $H$\hss}}%
    \graphtemp=.5ex\advance\graphtemp by 0.589in
    \rlap{\kern 5.019in\lower\graphtemp\hbox to 0pt{\hss $\bu$\hss}}%
    \graphtemp=.5ex\advance\graphtemp by 0.589in
    \rlap{\kern 5.308in\lower\graphtemp\hbox to 0pt{\hss $\bu$\hss}}%
    \graphtemp=.5ex\advance\graphtemp by 0.878in
    \rlap{\kern 4.731in\lower\graphtemp\hbox to 0pt{\hss $\bu$\hss}}%
    \graphtemp=.5ex\advance\graphtemp by 0.878in
    \rlap{\kern 5.596in\lower\graphtemp\hbox to 0pt{\hss $\bu$\hss}}%
    \graphtemp=.5ex\advance\graphtemp by 0.301in
    \rlap{\kern 4.731in\lower\graphtemp\hbox to 0pt{\hss $\bu$\hss}}%
    \graphtemp=.5ex\advance\graphtemp by 0.301in
    \rlap{\kern 5.596in\lower\graphtemp\hbox to 0pt{\hss $\bu$\hss}}%
    \graphtemp=.5ex\advance\graphtemp by 0.878in
    \rlap{\kern 4.442in\lower\graphtemp\hbox to 0pt{\hss $\bu$\hss}}%
    \graphtemp=.5ex\advance\graphtemp by 0.301in
    \rlap{\kern 4.442in\lower\graphtemp\hbox to 0pt{\hss $\bu$\hss}}%
    \graphtemp=.5ex\advance\graphtemp by 0.301in
    \rlap{\kern 4.442in\lower\graphtemp\hbox to 0pt{\hss $\bu$\hss}}%
    \special{pa 5019 589}%
    \special{pa 5308 589}%
    \special{fp}%
    \special{pa 5019 589}%
    \special{pa 5596 878}%
    \special{fp}%
    \special{pa 5019 589}%
    \special{pa 4731 301}%
    \special{fp}%
    \special{pa 5308 589}%
    \special{pa 5019 589}%
    \special{fp}%
    \special{pa 5308 589}%
    \special{pa 4731 878}%
    \special{fp}%
    \special{pa 5308 589}%
    \special{pa 5596 301}%
    \special{fp}%
    \special{pn 8}%
    \special{ar 5163 462 462 462 -2.786171 -0.355421}%
    \special{ar 5163 717 462 462 0.355421 2.786171}%
    \special{ar 4535 589 303 303 1.880641 4.402544}%
    \special{pn 11}%
    \special{pa 4731 878}%
    \special{pa 4442 878}%
    \special{fp}%
    \special{pa 4731 301}%
    \special{pa 4442 301}%
    \special{fp}%
    \graphtemp=.5ex\advance\graphtemp by 0.647in
    \rlap{\kern 4.904in\lower\graphtemp\hbox to 0pt{\hss $x'$\hss}}%
    \graphtemp=.5ex\advance\graphtemp by 0.988in
    \rlap{\kern 4.649in\lower\graphtemp\hbox to 0pt{\hss $b$\hss}}%
    \graphtemp=.5ex\advance\graphtemp by 0.219in
    \rlap{\kern 4.678in\lower\graphtemp\hbox to 0pt{\hss $a$\hss}}%
    \graphtemp=.5ex\advance\graphtemp by 0.647in
    \rlap{\kern 5.423in\lower\graphtemp\hbox to 0pt{\hss $y'$\hss}}%
    \graphtemp=.5ex\advance\graphtemp by 0.959in
    \rlap{\kern 5.678in\lower\graphtemp\hbox to 0pt{\hss $v$\hss}}%
    \graphtemp=.5ex\advance\graphtemp by 0.219in
    \rlap{\kern 5.678in\lower\graphtemp\hbox to 0pt{\hss $u$\hss}}%
    \graphtemp=.5ex\advance\graphtemp by 0.301in
    \rlap{\kern 5.163in\lower\graphtemp\hbox to 0pt{\hss $Q''$\hss}}%
    \graphtemp=.5ex\advance\graphtemp by 0.936in
    \rlap{\kern 5.163in\lower\graphtemp\hbox to 0pt{\hss $R''$\hss}}%
    \graphtemp=.5ex\advance\graphtemp by 0.589in
    \rlap{\kern 5.769in\lower\graphtemp\hbox to 0pt{\hss $H$\hss}}%
    \graphtemp=.5ex\advance\graphtemp by 0.589in
    \rlap{\kern 4.615in\lower\graphtemp\hbox to 0pt{\hss $C'$\hss}}%
    \hbox{\vrule depth1.179in width0pt height 0pt}%
    \kern 6.000in
  }%
}%
}
\caption{Three short cycles at a vertex.}\label{threecyc}
\end{figure}

\begin{lemma}\label{step7}
If $\la w,x,y,z\ra$ is a path in a $g$-cycle, and $wx$ lies in a $g'$-cycle,
then $yz$ also lies in a $g'$-cycle.
\end{lemma}
\begin{proof}
Let $C$ be the $g$-cycle in $G$ containing $\la w,x,y,z\ra$, and let
$D$ be the $g'$-cycle containing $wx$.
Let $a$ be the neighbor of $y$ outside $C$, and let $N_G(a)=\{y,b,c\}$.
Let $H$ be an alternative reconstruction from $G-\{y,a\}$, with $y'$ and $a'$
being the missing vertices.  We have $y'a'\in E(H)$, and we may label $y'$ and
$a'$ so that $y'z,a'x\in E(H)$.  Also, label $b$ and $c$ so that
$a'b,y'c\in E(H)$ (see Figure~\ref{ggcyc}).  Note that replacing $\la x,y,z\ra$
in $C$ with $\la x,a',y',z\ra$ yields a $g'$-cycle $C'$ in $H$.

Since $C$ is a $g$-cycle in $G$ containing exactly one of $\{y,a\}$, in $H$
there must be a $g$-cycle $D'$ containing exactly one of $\{y',a'\}$.
It must contain $\la x,a',b\ra$ or $\la z,y',c\ra$.  In the first case,
$x$ in $H$ lies in the $g$-cycle $D'$ and $g'$-cycles $C'$ and $D$, forbidden
by Lemma~\ref{step6}.  In the second case, replacing $\la z,y',c\ra$ in $D'$
with $\la z,y,a,c\ra$ yields the desired $g'$-cycle in $G$ through $yz$.
\end{proof}

\vspace{-1pc}
\begin{figure}[hbt]
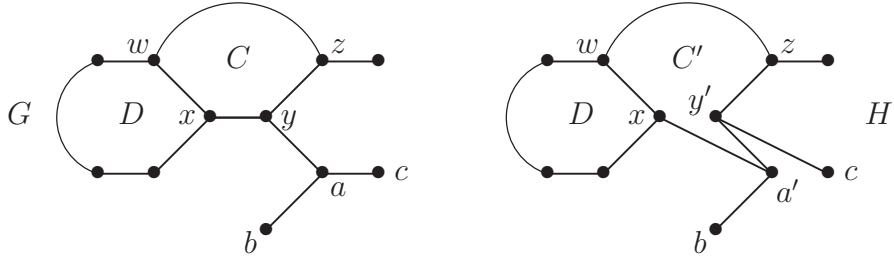

\gpic{
\expandafter\ifx\csname graph\endcsname\relax \csname newbox\endcsname\graph\fi
\expandafter\ifx\csname graphtemp\endcsname\relax \csname newdimen\endcsname\graphtemp\fi
\setbox\graph=\vtop{\vskip 0pt\hbox{%
    \graphtemp=.5ex\advance\graphtemp by 0.601in
    \rlap{\kern 1.000in\lower\graphtemp\hbox to 0pt{\hss $\bu$\hss}}%
    \graphtemp=.5ex\advance\graphtemp by 0.601in
    \rlap{\kern 1.294in\lower\graphtemp\hbox to 0pt{\hss $\bu$\hss}}%
    \graphtemp=.5ex\advance\graphtemp by 0.895in
    \rlap{\kern 0.706in\lower\graphtemp\hbox to 0pt{\hss $\bu$\hss}}%
    \graphtemp=.5ex\advance\graphtemp by 0.895in
    \rlap{\kern 1.588in\lower\graphtemp\hbox to 0pt{\hss $\bu$\hss}}%
    \graphtemp=.5ex\advance\graphtemp by 0.307in
    \rlap{\kern 0.706in\lower\graphtemp\hbox to 0pt{\hss $\bu$\hss}}%
    \graphtemp=.5ex\advance\graphtemp by 0.307in
    \rlap{\kern 1.588in\lower\graphtemp\hbox to 0pt{\hss $\bu$\hss}}%
    \graphtemp=.5ex\advance\graphtemp by 0.895in
    \rlap{\kern 0.412in\lower\graphtemp\hbox to 0pt{\hss $\bu$\hss}}%
    \graphtemp=.5ex\advance\graphtemp by 0.895in
    \rlap{\kern 1.882in\lower\graphtemp\hbox to 0pt{\hss $\bu$\hss}}%
    \graphtemp=.5ex\advance\graphtemp by 0.307in
    \rlap{\kern 0.412in\lower\graphtemp\hbox to 0pt{\hss $\bu$\hss}}%
    \graphtemp=.5ex\advance\graphtemp by 1.189in
    \rlap{\kern 1.294in\lower\graphtemp\hbox to 0pt{\hss $\bu$\hss}}%
    \graphtemp=.5ex\advance\graphtemp by 0.307in
    \rlap{\kern 1.882in\lower\graphtemp\hbox to 0pt{\hss $\bu$\hss}}%
    \special{pn 11}%
    \special{pa 1588 895}%
    \special{pa 1294 1189}%
    \special{fp}%
    \special{pa 1588 895}%
    \special{pa 1882 895}%
    \special{fp}%
    \special{pa 1588 307}%
    \special{pa 1882 307}%
    \special{fp}%
    \special{pa 1000 601}%
    \special{pa 1294 601}%
    \special{fp}%
    \special{pa 1000 601}%
    \special{pa 706 895}%
    \special{fp}%
    \special{pa 1000 601}%
    \special{pa 706 307}%
    \special{fp}%
    \special{pa 1294 601}%
    \special{pa 1000 601}%
    \special{fp}%
    \special{pa 1294 601}%
    \special{pa 1588 895}%
    \special{fp}%
    \special{pa 1294 601}%
    \special{pa 1588 307}%
    \special{fp}%
    \special{pn 8}%
    \special{ar 1147 471 471 471 -2.786171 -0.355421}%
    \special{ar 506 601 309 309 1.880641 4.402544}%
    \special{pn 11}%
    \special{pa 706 895}%
    \special{pa 412 895}%
    \special{fp}%
    \special{pa 706 307}%
    \special{pa 412 307}%
    \special{fp}%
    \graphtemp=.5ex\advance\graphtemp by 0.601in
    \rlap{\kern 0.882in\lower\graphtemp\hbox to 0pt{\hss $x$\hss}}%
    \graphtemp=.5ex\advance\graphtemp by 1.272in
    \rlap{\kern 1.211in\lower\graphtemp\hbox to 0pt{\hss $b$\hss}}%
    \graphtemp=.5ex\advance\graphtemp by 0.224in
    \rlap{\kern 0.623in\lower\graphtemp\hbox to 0pt{\hss $w$\hss}}%
    \graphtemp=.5ex\advance\graphtemp by 0.601in
    \rlap{\kern 1.412in\lower\graphtemp\hbox to 0pt{\hss $y$\hss}}%
    \graphtemp=.5ex\advance\graphtemp by 0.978in
    \rlap{\kern 1.671in\lower\graphtemp\hbox to 0pt{\hss $a$\hss}}%
    \graphtemp=.5ex\advance\graphtemp by 0.224in
    \rlap{\kern 1.671in\lower\graphtemp\hbox to 0pt{\hss $z$\hss}}%
    \graphtemp=.5ex\advance\graphtemp by 0.895in
    \rlap{\kern 2.000in\lower\graphtemp\hbox to 0pt{\hss $c$\hss}}%
    \graphtemp=.5ex\advance\graphtemp by 0.307in
    \rlap{\kern 1.147in\lower\graphtemp\hbox to 0pt{\hss $C$\hss}}%
    \graphtemp=.5ex\advance\graphtemp by 0.601in
    \rlap{\kern 0.588in\lower\graphtemp\hbox to 0pt{\hss $D$\hss}}%
    \graphtemp=.5ex\advance\graphtemp by 0.601in
    \rlap{\kern 0.000in\lower\graphtemp\hbox to 0pt{\hss $G$\hss}}%
    \graphtemp=.5ex\advance\graphtemp by 0.601in
    \rlap{\kern 3.353in\lower\graphtemp\hbox to 0pt{\hss $\bu$\hss}}%
    \graphtemp=.5ex\advance\graphtemp by 0.601in
    \rlap{\kern 3.647in\lower\graphtemp\hbox to 0pt{\hss $\bu$\hss}}%
    \graphtemp=.5ex\advance\graphtemp by 0.895in
    \rlap{\kern 3.059in\lower\graphtemp\hbox to 0pt{\hss $\bu$\hss}}%
    \graphtemp=.5ex\advance\graphtemp by 0.895in
    \rlap{\kern 3.941in\lower\graphtemp\hbox to 0pt{\hss $\bu$\hss}}%
    \graphtemp=.5ex\advance\graphtemp by 0.307in
    \rlap{\kern 3.059in\lower\graphtemp\hbox to 0pt{\hss $\bu$\hss}}%
    \graphtemp=.5ex\advance\graphtemp by 0.307in
    \rlap{\kern 3.941in\lower\graphtemp\hbox to 0pt{\hss $\bu$\hss}}%
    \graphtemp=.5ex\advance\graphtemp by 0.895in
    \rlap{\kern 2.765in\lower\graphtemp\hbox to 0pt{\hss $\bu$\hss}}%
    \graphtemp=.5ex\advance\graphtemp by 0.895in
    \rlap{\kern 4.235in\lower\graphtemp\hbox to 0pt{\hss $\bu$\hss}}%
    \graphtemp=.5ex\advance\graphtemp by 0.307in
    \rlap{\kern 2.765in\lower\graphtemp\hbox to 0pt{\hss $\bu$\hss}}%
    \graphtemp=.5ex\advance\graphtemp by 1.189in
    \rlap{\kern 3.647in\lower\graphtemp\hbox to 0pt{\hss $\bu$\hss}}%
    \graphtemp=.5ex\advance\graphtemp by 0.307in
    \rlap{\kern 4.235in\lower\graphtemp\hbox to 0pt{\hss $\bu$\hss}}%
    \special{pa 3941 895}%
    \special{pa 3647 1189}%
    \special{fp}%
    \special{pa 3941 307}%
    \special{pa 4235 307}%
    \special{fp}%
    \special{pa 3353 601}%
    \special{pa 3941 895}%
    \special{fp}%
    \special{pa 3353 601}%
    \special{pa 3059 895}%
    \special{fp}%
    \special{pa 3353 601}%
    \special{pa 3059 307}%
    \special{fp}%
    \special{pa 3647 601}%
    \special{pa 4235 895}%
    \special{fp}%
    \special{pa 3647 601}%
    \special{pa 3941 895}%
    \special{fp}%
    \special{pa 3647 601}%
    \special{pa 3941 307}%
    \special{fp}%
    \special{pn 8}%
    \special{ar 3500 471 471 471 -2.786171 -0.355421}%
    \special{ar 2859 601 309 309 1.880641 4.402544}%
    \special{pn 11}%
    \special{pa 3059 895}%
    \special{pa 2765 895}%
    \special{fp}%
    \special{pa 3059 307}%
    \special{pa 2765 307}%
    \special{fp}%
    \graphtemp=.5ex\advance\graphtemp by 0.601in
    \rlap{\kern 3.235in\lower\graphtemp\hbox to 0pt{\hss $x$\hss}}%
    \graphtemp=.5ex\advance\graphtemp by 1.272in
    \rlap{\kern 3.564in\lower\graphtemp\hbox to 0pt{\hss $b$\hss}}%
    \graphtemp=.5ex\advance\graphtemp by 0.224in
    \rlap{\kern 2.976in\lower\graphtemp\hbox to 0pt{\hss $w$\hss}}%
    \graphtemp=.5ex\advance\graphtemp by 0.518in
    \rlap{\kern 3.564in\lower\graphtemp\hbox to 0pt{\hss $y'$\hss}}%
    \graphtemp=.5ex\advance\graphtemp by 1.008in
    \rlap{\kern 4.024in\lower\graphtemp\hbox to 0pt{\hss $a'$\hss}}%
    \graphtemp=.5ex\advance\graphtemp by 0.224in
    \rlap{\kern 4.024in\lower\graphtemp\hbox to 0pt{\hss $z$\hss}}%
    \graphtemp=.5ex\advance\graphtemp by 0.895in
    \rlap{\kern 4.353in\lower\graphtemp\hbox to 0pt{\hss $c$\hss}}%
    \graphtemp=.5ex\advance\graphtemp by 0.307in
    \rlap{\kern 3.500in\lower\graphtemp\hbox to 0pt{\hss $C'$\hss}}%
    \graphtemp=.5ex\advance\graphtemp by 0.601in
    \rlap{\kern 2.941in\lower\graphtemp\hbox to 0pt{\hss $D$\hss}}%
    \graphtemp=.5ex\advance\graphtemp by 0.601in
    \rlap{\kern 4.500in\lower\graphtemp\hbox to 0pt{\hss $H$\hss}}%
    \hbox{\vrule depth1.307in width0pt height 0pt}%
    \kern 4.500in
  }%
}%
}
\caption{A $g$-cycle and a $g'$-cycle with a common edge.}\label{ggcyc}
\end{figure}

\begin{lemma}\label{step8}
No $g$-cycle and $g'$-cycle share an edge.
\end{lemma}
\begin{proof}
Let $wx$ be such an edge in $G$, shared by the $g'$-cycle $D$ and the
$g$-cycle $C$ containing the path $\la w,x,y,z,c\ra$.  By Lemma~\ref{step7},
$yz$ lies in a $g'$-cycle $B$ in $G$.

Since $w$ and $y$ lie on a $g$-cycle, from $G-\{w,y\}$ there is only one
alternative reconstruction $H$, in which by Lemma~\ref{shortest} the $g$-cycle
$C'$ through the missing vertices $w'$ and $y'$ is obtained from $C$ by
replacing $w$ with $w'$ and $y$ with $y'$.  Thus also $w'a,y'b\in E(H)$,
where $a$ and $b$ are the neighbors outside $C$ of $y$ and $w$ in $G$,
respectively (see Figure~\ref{noshare}).

Replacing $\la b,w,x\ra$ in $D$ with $\la b,y',x\ra$ yields a $g'$-cycle $D'$
in $H$.  Since $C'$ and $D'$ share the edge $xy'$, by Lemma~\ref{step7}
the edge $zc$ lies in a $g'$-cycle $Q$ in $H$.  Since shorter cycles and
spliced cycles must be avoided, $Q$ avoids $y'$ and $w'$.  Hence $Q$ appears
also in $G$.  Now in $G$ the vertex $z$ appears in the $g$-cycle $C$ and
$g'$-cycles $B$ and $Q$, which is forbidden by Lemma~\ref{step6}.
\end{proof}

\vspace{-1pc}
\begin{figure}[hbt]
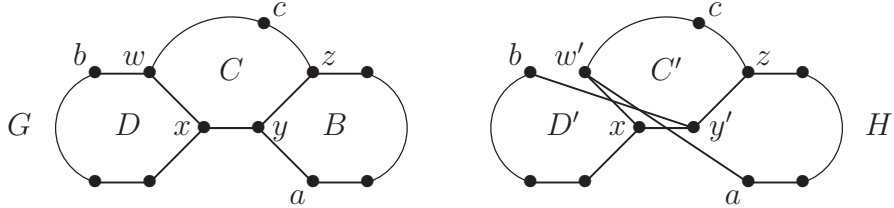

\gpic{
\expandafter\ifx\csname graph\endcsname\relax \csname newbox\endcsname\graph\fi
\expandafter\ifx\csname graphtemp\endcsname\relax \csname newdimen\endcsname\graphtemp\fi
\setbox\graph=\vtop{\vskip 0pt\hbox{%
    \graphtemp=.5ex\advance\graphtemp by 0.655in
    \rlap{\kern 0.968in\lower\graphtemp\hbox to 0pt{\hss $\bu$\hss}}%
    \graphtemp=.5ex\advance\graphtemp by 0.655in
    \rlap{\kern 1.253in\lower\graphtemp\hbox to 0pt{\hss $\bu$\hss}}%
    \graphtemp=.5ex\advance\graphtemp by 0.940in
    \rlap{\kern 0.684in\lower\graphtemp\hbox to 0pt{\hss $\bu$\hss}}%
    \graphtemp=.5ex\advance\graphtemp by 0.940in
    \rlap{\kern 1.538in\lower\graphtemp\hbox to 0pt{\hss $\bu$\hss}}%
    \graphtemp=.5ex\advance\graphtemp by 0.370in
    \rlap{\kern 0.684in\lower\graphtemp\hbox to 0pt{\hss $\bu$\hss}}%
    \graphtemp=.5ex\advance\graphtemp by 0.370in
    \rlap{\kern 1.538in\lower\graphtemp\hbox to 0pt{\hss $\bu$\hss}}%
    \graphtemp=.5ex\advance\graphtemp by 0.940in
    \rlap{\kern 0.399in\lower\graphtemp\hbox to 0pt{\hss $\bu$\hss}}%
    \graphtemp=.5ex\advance\graphtemp by 0.940in
    \rlap{\kern 1.823in\lower\graphtemp\hbox to 0pt{\hss $\bu$\hss}}%
    \graphtemp=.5ex\advance\graphtemp by 0.370in
    \rlap{\kern 0.399in\lower\graphtemp\hbox to 0pt{\hss $\bu$\hss}}%
    \graphtemp=.5ex\advance\graphtemp by 0.114in
    \rlap{\kern 1.282in\lower\graphtemp\hbox to 0pt{\hss $\bu$\hss}}%
    \graphtemp=.5ex\advance\graphtemp by 0.033in
    \rlap{\kern 1.362in\lower\graphtemp\hbox to 0pt{\hss $c$\hss}}%
    \graphtemp=.5ex\advance\graphtemp by 0.370in
    \rlap{\kern 1.823in\lower\graphtemp\hbox to 0pt{\hss $\bu$\hss}}%
    \special{pn 11}%
    \special{pa 1538 940}%
    \special{pa 1823 940}%
    \special{fp}%
    \special{pa 1538 370}%
    \special{pa 1823 370}%
    \special{fp}%
    \special{pa 684 940}%
    \special{pa 399 940}%
    \special{fp}%
    \special{pa 684 370}%
    \special{pa 399 370}%
    \special{fp}%
    \special{pa 968 655}%
    \special{pa 1253 655}%
    \special{fp}%
    \special{pa 968 655}%
    \special{pa 684 940}%
    \special{fp}%
    \special{pa 968 655}%
    \special{pa 684 370}%
    \special{fp}%
    \special{pa 1253 655}%
    \special{pa 968 655}%
    \special{fp}%
    \special{pa 1253 655}%
    \special{pa 1538 940}%
    \special{fp}%
    \special{pa 1253 655}%
    \special{pa 1538 370}%
    \special{fp}%
    \special{pn 8}%
    \special{ar 1111 529 456 456 -2.786171 -0.355421}%
    \special{ar 490 655 299 299 1.880641 4.402544}%
    \special{ar 1732 655 299 299 -1.260952 1.260952}%
    \graphtemp=.5ex\advance\graphtemp by 0.655in
    \rlap{\kern 0.854in\lower\graphtemp\hbox to 0pt{\hss $x$\hss}}%
    \graphtemp=.5ex\advance\graphtemp by 0.290in
    \rlap{\kern 0.318in\lower\graphtemp\hbox to 0pt{\hss $b$\hss}}%
    \graphtemp=.5ex\advance\graphtemp by 0.290in
    \rlap{\kern 0.603in\lower\graphtemp\hbox to 0pt{\hss $w$\hss}}%
    \graphtemp=.5ex\advance\graphtemp by 0.655in
    \rlap{\kern 1.652in\lower\graphtemp\hbox to 0pt{\hss $B$\hss}}%
    \graphtemp=.5ex\advance\graphtemp by 0.655in
    \rlap{\kern 1.367in\lower\graphtemp\hbox to 0pt{\hss $y$\hss}}%
    \graphtemp=.5ex\advance\graphtemp by 1.020in
    \rlap{\kern 1.457in\lower\graphtemp\hbox to 0pt{\hss $a$\hss}}%
    \graphtemp=.5ex\advance\graphtemp by 0.290in
    \rlap{\kern 1.619in\lower\graphtemp\hbox to 0pt{\hss $z$\hss}}%
    \graphtemp=.5ex\advance\graphtemp by 0.370in
    \rlap{\kern 1.111in\lower\graphtemp\hbox to 0pt{\hss $C$\hss}}%
    \graphtemp=.5ex\advance\graphtemp by 0.655in
    \rlap{\kern 0.570in\lower\graphtemp\hbox to 0pt{\hss $D$\hss}}%
    \graphtemp=.5ex\advance\graphtemp by 0.655in
    \rlap{\kern 0.000in\lower\graphtemp\hbox to 0pt{\hss $G$\hss}}%
    \graphtemp=.5ex\advance\graphtemp by 0.655in
    \rlap{\kern 3.247in\lower\graphtemp\hbox to 0pt{\hss $\bu$\hss}}%
    \graphtemp=.5ex\advance\graphtemp by 0.655in
    \rlap{\kern 3.532in\lower\graphtemp\hbox to 0pt{\hss $\bu$\hss}}%
    \graphtemp=.5ex\advance\graphtemp by 0.940in
    \rlap{\kern 2.962in\lower\graphtemp\hbox to 0pt{\hss $\bu$\hss}}%
    \graphtemp=.5ex\advance\graphtemp by 0.940in
    \rlap{\kern 3.816in\lower\graphtemp\hbox to 0pt{\hss $\bu$\hss}}%
    \graphtemp=.5ex\advance\graphtemp by 0.370in
    \rlap{\kern 2.962in\lower\graphtemp\hbox to 0pt{\hss $\bu$\hss}}%
    \graphtemp=.5ex\advance\graphtemp by 0.370in
    \rlap{\kern 3.816in\lower\graphtemp\hbox to 0pt{\hss $\bu$\hss}}%
    \graphtemp=.5ex\advance\graphtemp by 0.940in
    \rlap{\kern 2.677in\lower\graphtemp\hbox to 0pt{\hss $\bu$\hss}}%
    \graphtemp=.5ex\advance\graphtemp by 0.940in
    \rlap{\kern 4.101in\lower\graphtemp\hbox to 0pt{\hss $\bu$\hss}}%
    \graphtemp=.5ex\advance\graphtemp by 0.370in
    \rlap{\kern 2.677in\lower\graphtemp\hbox to 0pt{\hss $\bu$\hss}}%
    \graphtemp=.5ex\advance\graphtemp by 0.114in
    \rlap{\kern 3.560in\lower\graphtemp\hbox to 0pt{\hss $\bu$\hss}}%
    \graphtemp=.5ex\advance\graphtemp by 0.033in
    \rlap{\kern 3.641in\lower\graphtemp\hbox to 0pt{\hss $c$\hss}}%
    \graphtemp=.5ex\advance\graphtemp by 0.370in
    \rlap{\kern 4.101in\lower\graphtemp\hbox to 0pt{\hss $\bu$\hss}}%
    \special{pn 11}%
    \special{pa 3816 370}%
    \special{pa 4101 370}%
    \special{fp}%
    \special{pa 2962 940}%
    \special{pa 2677 940}%
    \special{fp}%
    \special{pa 2962 370}%
    \special{pa 3816 940}%
    \special{fp}%
    \special{pa 3816 940}%
    \special{pa 4101 940}%
    \special{fp}%
    \special{pa 3247 655}%
    \special{pa 3532 655}%
    \special{fp}%
    \special{pa 3247 655}%
    \special{pa 2962 940}%
    \special{fp}%
    \special{pa 3247 655}%
    \special{pa 2962 370}%
    \special{fp}%
    \special{pa 3532 655}%
    \special{pa 3247 655}%
    \special{fp}%
    \special{pa 3532 655}%
    \special{pa 2677 370}%
    \special{fp}%
    \special{pa 3532 655}%
    \special{pa 3816 370}%
    \special{fp}%
    \special{pn 8}%
    \special{ar 3389 529 456 456 -2.786171 -0.355421}%
    \special{ar 2768 655 299 299 1.880641 4.402544}%
    \special{ar 4010 655 299 299 -1.260952 1.260952}%
    \graphtemp=.5ex\advance\graphtemp by 0.655in
    \rlap{\kern 3.133in\lower\graphtemp\hbox to 0pt{\hss $x$\hss}}%
    \graphtemp=.5ex\advance\graphtemp by 0.290in
    \rlap{\kern 2.597in\lower\graphtemp\hbox to 0pt{\hss $b$\hss}}%
    \graphtemp=.5ex\advance\graphtemp by 0.290in
    \rlap{\kern 2.881in\lower\graphtemp\hbox to 0pt{\hss $w'$\hss}}%
    \graphtemp=.5ex\advance\graphtemp by 0.655in
    \rlap{\kern 3.674in\lower\graphtemp\hbox to 0pt{\hss $y'$\hss}}%
    \graphtemp=.5ex\advance\graphtemp by 1.020in
    \rlap{\kern 3.736in\lower\graphtemp\hbox to 0pt{\hss $a$\hss}}%
    \graphtemp=.5ex\advance\graphtemp by 0.290in
    \rlap{\kern 3.897in\lower\graphtemp\hbox to 0pt{\hss $z$\hss}}%
    \graphtemp=.5ex\advance\graphtemp by 0.370in
    \rlap{\kern 3.389in\lower\graphtemp\hbox to 0pt{\hss $C'$\hss}}%
    \graphtemp=.5ex\advance\graphtemp by 0.655in
    \rlap{\kern 2.848in\lower\graphtemp\hbox to 0pt{\hss $D'$\hss}}%
    \graphtemp=.5ex\advance\graphtemp by 0.655in
    \rlap{\kern 4.500in\lower\graphtemp\hbox to 0pt{\hss $H$\hss}}%
    \hbox{\vrule depth1.054in width0pt height 0pt}%
    \kern 4.500in
  }%
}%
}
\caption{Another $g$-cycle and a $g'$-cycle with a common edge.}\label{noshare}
\end{figure}

\begin{lemma}\label{step9}
Every $g$-cycle shares an edge with some $g'$-cycle.
\end{lemma}
\begin{proof}
Let $\la a,x,b\ra$ be a path along a $g$-cycle $C$ in $G$.
Let $y$ be the neighbor of $x$ outside $C$, with $N_G(y)=\{x,u,v\}$.
Let $H$ be an alternative reconstruction from $G-\{x,y\}$, with missing
vertices $x'$ and $y'$.  As usual, $x'y'\in E(H)$, and by symmetry we may
assume $x'a,y'b\in E(H)$.  We may also choose the labels $u$ and $v$ so that
$x'v,y'u\in E(H)$ (see Figure~\ref{share}).

Since $C$ is a $g$-cycle in $G$ containing exactly one of $x$ and $y$,
in $H$ there must be a $g$-cycle $C'$ containing exactly one of $x'$ and $y'$.
Such a cycle must contain $\la a,x',v\ra$ or $\la b,y',u\ra$.
Replacing this path in $C'$ with $\la a,x,y,v\ra$ or $\la b,x,y,u\ra$,
respectively, yields a $g'$-cycle in $G$ that shares an edge with $C$.
\end{proof}

\vspace{-1pc}
\begin{figure}[hbt]
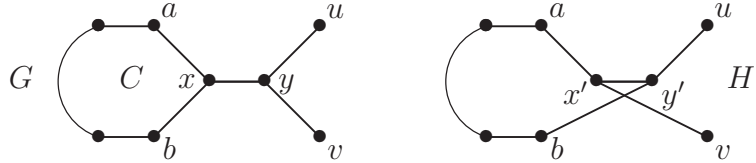

\gpic{
\expandafter\ifx\csname graph\endcsname\relax \csname newbox\endcsname\graph\fi
\expandafter\ifx\csname graphtemp\endcsname\relax \csname newdimen\endcsname\graphtemp\fi
\setbox\graph=\vtop{\vskip 0pt\hbox{%
    \graphtemp=.5ex\advance\graphtemp by 0.406in
    \rlap{\kern 0.986in\lower\graphtemp\hbox to 0pt{\hss $\bu$\hss}}%
    \graphtemp=.5ex\advance\graphtemp by 0.406in
    \rlap{\kern 1.275in\lower\graphtemp\hbox to 0pt{\hss $\bu$\hss}}%
    \graphtemp=.5ex\advance\graphtemp by 0.696in
    \rlap{\kern 0.696in\lower\graphtemp\hbox to 0pt{\hss $\bu$\hss}}%
    \graphtemp=.5ex\advance\graphtemp by 0.696in
    \rlap{\kern 1.565in\lower\graphtemp\hbox to 0pt{\hss $\bu$\hss}}%
    \graphtemp=.5ex\advance\graphtemp by 0.116in
    \rlap{\kern 0.696in\lower\graphtemp\hbox to 0pt{\hss $\bu$\hss}}%
    \graphtemp=.5ex\advance\graphtemp by 0.116in
    \rlap{\kern 1.565in\lower\graphtemp\hbox to 0pt{\hss $\bu$\hss}}%
    \graphtemp=.5ex\advance\graphtemp by 0.696in
    \rlap{\kern 0.406in\lower\graphtemp\hbox to 0pt{\hss $\bu$\hss}}%
    \graphtemp=.5ex\advance\graphtemp by 0.116in
    \rlap{\kern 0.406in\lower\graphtemp\hbox to 0pt{\hss $\bu$\hss}}%
    \graphtemp=.5ex\advance\graphtemp by 0.116in
    \rlap{\kern 0.406in\lower\graphtemp\hbox to 0pt{\hss $\bu$\hss}}%
    \special{pn 11}%
    \special{pa 696 696}%
    \special{pa 406 696}%
    \special{fp}%
    \special{pa 696 116}%
    \special{pa 406 116}%
    \special{fp}%
    \special{pa 986 406}%
    \special{pa 1275 406}%
    \special{fp}%
    \special{pa 986 406}%
    \special{pa 696 696}%
    \special{fp}%
    \special{pa 986 406}%
    \special{pa 696 116}%
    \special{fp}%
    \special{pa 1275 406}%
    \special{pa 986 406}%
    \special{fp}%
    \special{pa 1275 406}%
    \special{pa 1565 696}%
    \special{fp}%
    \special{pa 1275 406}%
    \special{pa 1565 116}%
    \special{fp}%
    \special{pn 8}%
    \special{ar 499 406 304 304 1.880641 4.402544}%
    \graphtemp=.5ex\advance\graphtemp by 0.406in
    \rlap{\kern 0.870in\lower\graphtemp\hbox to 0pt{\hss $x$\hss}}%
    \graphtemp=.5ex\advance\graphtemp by 0.778in
    \rlap{\kern 0.778in\lower\graphtemp\hbox to 0pt{\hss $b$\hss}}%
    \graphtemp=.5ex\advance\graphtemp by 0.034in
    \rlap{\kern 0.778in\lower\graphtemp\hbox to 0pt{\hss $a$\hss}}%
    \graphtemp=.5ex\advance\graphtemp by 0.406in
    \rlap{\kern 1.391in\lower\graphtemp\hbox to 0pt{\hss $y$\hss}}%
    \graphtemp=.5ex\advance\graphtemp by 0.778in
    \rlap{\kern 1.647in\lower\graphtemp\hbox to 0pt{\hss $v$\hss}}%
    \graphtemp=.5ex\advance\graphtemp by 0.034in
    \rlap{\kern 1.647in\lower\graphtemp\hbox to 0pt{\hss $u$\hss}}%
    \graphtemp=.5ex\advance\graphtemp by 0.116in
    \rlap{\kern 1.130in\lower\graphtemp\hbox to 0pt{\hss $~$\hss}}%
    \graphtemp=.5ex\advance\graphtemp by 0.406in
    \rlap{\kern 0.580in\lower\graphtemp\hbox to 0pt{\hss $C$\hss}}%
    \graphtemp=.5ex\advance\graphtemp by 0.406in
    \rlap{\kern 0.000in\lower\graphtemp\hbox to 0pt{\hss $G$\hss}}%
    \graphtemp=.5ex\advance\graphtemp by 0.406in
    \rlap{\kern 3.014in\lower\graphtemp\hbox to 0pt{\hss $\bu$\hss}}%
    \graphtemp=.5ex\advance\graphtemp by 0.406in
    \rlap{\kern 3.304in\lower\graphtemp\hbox to 0pt{\hss $\bu$\hss}}%
    \graphtemp=.5ex\advance\graphtemp by 0.696in
    \rlap{\kern 2.725in\lower\graphtemp\hbox to 0pt{\hss $\bu$\hss}}%
    \graphtemp=.5ex\advance\graphtemp by 0.696in
    \rlap{\kern 3.594in\lower\graphtemp\hbox to 0pt{\hss $\bu$\hss}}%
    \graphtemp=.5ex\advance\graphtemp by 0.116in
    \rlap{\kern 2.725in\lower\graphtemp\hbox to 0pt{\hss $\bu$\hss}}%
    \graphtemp=.5ex\advance\graphtemp by 0.116in
    \rlap{\kern 3.594in\lower\graphtemp\hbox to 0pt{\hss $\bu$\hss}}%
    \graphtemp=.5ex\advance\graphtemp by 0.696in
    \rlap{\kern 2.435in\lower\graphtemp\hbox to 0pt{\hss $\bu$\hss}}%
    \graphtemp=.5ex\advance\graphtemp by 0.116in
    \rlap{\kern 2.435in\lower\graphtemp\hbox to 0pt{\hss $\bu$\hss}}%
    \graphtemp=.5ex\advance\graphtemp by 0.116in
    \rlap{\kern 2.435in\lower\graphtemp\hbox to 0pt{\hss $\bu$\hss}}%
    \special{pn 11}%
    \special{pa 2725 696}%
    \special{pa 2435 696}%
    \special{fp}%
    \special{pa 2725 116}%
    \special{pa 2435 116}%
    \special{fp}%
    \special{pa 3014 406}%
    \special{pa 3304 406}%
    \special{fp}%
    \special{pa 3014 406}%
    \special{pa 3594 696}%
    \special{fp}%
    \special{pa 3014 406}%
    \special{pa 2725 116}%
    \special{fp}%
    \special{pa 3304 406}%
    \special{pa 3014 406}%
    \special{fp}%
    \special{pa 3304 406}%
    \special{pa 2725 696}%
    \special{fp}%
    \special{pa 3304 406}%
    \special{pa 3594 116}%
    \special{fp}%
    \special{pn 8}%
    \special{ar 2528 406 304 304 1.880641 4.402544}%
    \graphtemp=.5ex\advance\graphtemp by 0.488in
    \rlap{\kern 2.904in\lower\graphtemp\hbox to 0pt{\hss $x'$\hss}}%
    \graphtemp=.5ex\advance\graphtemp by 0.778in
    \rlap{\kern 2.807in\lower\graphtemp\hbox to 0pt{\hss $b$\hss}}%
    \graphtemp=.5ex\advance\graphtemp by 0.034in
    \rlap{\kern 2.807in\lower\graphtemp\hbox to 0pt{\hss $a$\hss}}%
    \graphtemp=.5ex\advance\graphtemp by 0.488in
    \rlap{\kern 3.415in\lower\graphtemp\hbox to 0pt{\hss $y'$\hss}}%
    \graphtemp=.5ex\advance\graphtemp by 0.778in
    \rlap{\kern 3.676in\lower\graphtemp\hbox to 0pt{\hss $v$\hss}}%
    \graphtemp=.5ex\advance\graphtemp by 0.034in
    \rlap{\kern 3.676in\lower\graphtemp\hbox to 0pt{\hss $u$\hss}}%
    \graphtemp=.5ex\advance\graphtemp by 0.116in
    \rlap{\kern 3.159in\lower\graphtemp\hbox to 0pt{\hss $~$\hss}}%
    \graphtemp=.5ex\advance\graphtemp by 0.406in
    \rlap{\kern 2.609in\lower\graphtemp\hbox to 0pt{\hss $~$\hss}}%
    \graphtemp=.5ex\advance\graphtemp by 0.406in
    \rlap{\kern 3.768in\lower\graphtemp\hbox to 0pt{\hss $H$\hss}}%
    \hbox{\vrule depth0.812in width0pt height 0pt}%
    \kern 4.000in
  }%
}%
}
\caption{Forcing a common edge.}\label{share}
\end{figure}

Lemmas~\ref{step8} and~\ref{step9} are contradictory.  Hence no $3$-regular
graph $G$ has an alternative reconstruction from its $(n-2)$-deck, which proves
Theorem~\ref{main}.

\bigskip
\centerline{\bf Acknowledgment}
The authors thank Martin Merker, Bojan Mohar, and Hehui Wu for their 
contributions to early discussions.

\end{document}